\documentclass[final,hidelinks,onefignum,onetabnum]{siamart220329}

\usepackage{lipsum}
\usepackage{amsfonts}
\usepackage{graphicx}
\usepackage{epstopdf}
\usepackage{algorithmicx}
\usepackage[noend]{algpseudocode}
\usepackage[algo2e,ruled,linesnumbered]{algorithm2e}
\ifpdf
  \DeclareGraphicsExtensions{.eps,.pdf,.png,.jpg}
\else
  \DeclareGraphicsExtensions{.eps}
\fi
\usepackage[font=small]{caption}
\usepackage{subcaption}

\usepackage{datetime}
\newdateformat{monthyeardate}{%
  \monthname[\THEMONTH], \THEYEAR}

\newsiamremark{remark}{Remark}
\newsiamremark{hypothesis}{Hypothesis}
\crefname{hypothesis}{Hypothesis}{Hypotheses}
\newsiamthm{claim}{Claim}
\Crefname{algocf}{Algorithm}{Algorithms}

\usepackage{enumitem}

\usepackage{amsmath}
\allowdisplaybreaks
\usepackage{amssymb}
\usepackage{commath}
\usepackage{mathtools}
\usepackage{bbm}
\usepackage{cases}
\usepackage{stackengine}

\newcommand{\push}[2]{(#1)_{\#}#2}
\newcommand\proj[0]{\textbf{Proj}}
\newcommand{\id}{\mathop{}\mathopen{}\mathrm{id}}
\def\defeq{\mathrel{\ensurestackMath{\stackon[1pt]{=}{\scriptscriptstyle\Delta}}}}

\newcommand{\R}{\mathbb{R}} 
\renewcommand{\H}{\mathcal{H}}
\renewcommand{\P}{\mathbb{P}}

\usepackage{tikz}
\usepackage{multirow}
\usepackage{soul}
\usepackage{qtree}
\usepackage{forest}

\headers{MMOT with Pairwise Cost}{Zhou and Parno}

\title{Efficient and Exact Multimarginal Optimal Transport with Pairwise Costs\thanks{\monthyeardate\today.\funding{ONR MURI \#N00014-20-1-2595.}
}}

\author{ 
Bohan Zhou\footnotemark[2]\thanks{Department of Mathematics, Dartmouth College, Hanover, NH, USA (\email{Bohan.Zhou@Dartmouth.edu}).}
\and
Matthew Parno\footnotemark[2]\thanks{Department of Mathematics, Dartmouth College, Hanover, NH, USA (\email{Matthew.D.Parno@Dartmouth.edu})}
}

\usepackage{amsopn}

\ifpdf
\hypersetup{
  pdftitle={Efficient and Exact MMOT with pairwise costs},
  pdfauthor={Bohan Zhou}
}
\fi

\begin{document}

\maketitle

\begin{abstract}
We address the numerical solution to multimarginal optimal transport (MMOT) with pairwise costs. MMOT, as a natural extension from the classical two-marginal optimal transport, has many important applications including image processing, density functional theory and machine learning, but lacks efficient and exact numerical methods. The popular entropy-regularized method may suffer numerical instability and blurring issues. Inspired by the back-and-forth method introduced by Jacobs and L\'{e}ger, we investigate MMOT problems with pairwise costs. We show that such problems have a graphical representation and leverage this structure to develop a new computationally gradient ascent algorithm to solve the dual formulation of such MMOT problems.  Our method produces accurate solutions which can be used for the regularization-free applications, including the computation of Wasserstein barycenters with high resolution imagery.
\end{abstract}



\section{Introduction}

Probability distributions are used throughout statistics, machine learning, and applied mathematics to model complex datasets and characterize uncertainty.  Quantitatively comparing distributions and identifying structure in the space of probability distributions are therefore fundamental components of many modern algorithms for data analysis.  Optimal transport (OT) provides a natural way of comparing two distributions by measuring how much effort is required to transform one distribution into another.  The solution of an OT problem provides both a distance, called the \textit{Wasserstein distance}, and a joint distribution, called the \textit{optimal coupling}, which describes the optimal mass allocation between marginal distributions. Multi-marginal optimal transport (MMOT), which is the focus of this work, provides a generalization of classic OT to problems with more than two marginal distributions.

The field of OT has existed since Monge in the late 18th century, but has reemerged over the last few decades as a powerful theoretical and computational tool in many areas.   
Applications can be found in fields as diverse as chemistry and materials science \cite{Peletier2009Partial, De2016Comparing, Xia2021Existence}, geophysics \cite{Yang2018Application,Parno2019Remote}, image processing \cite{Saumier2015OT, Benamou2015Iterative, Solomon2015Convolutional}, fluid dynamics \cite{Brenier1989LeastAction, Benamou2000Computational, Brenier2008Generalized}, and machine learning \cite{Frogner2015Learning, Courty2017Joint,Arjovsky2017Wasserstein,Genevay2018Learning, Caron2020Unsupervised}, to name just a few.  A key contributor to this surge is the development of efficient numerical methods for approximately solving OT problems. The concept of regularized optimal transport received renewed attention following \cite{Cuturi2013Sinkhorn}, where an entropy regularization term is added to the original optimal transport problem, to form a matrix scaling problem that can be efficiently solved with Sinkhorn iterations \cite{sinkhorn1967concerning}.  The result is an easy-to-compute, albeit approximate, solution to the optimal transport problem.  

Many refinements and extensions of regularized optimal transport have since been developed (e.g., \cite{Dessein2018Regularized,Cuturi2018Semidual,Schmitzer2019Stabilized}) including GPU-accelerated implementations \cite{Feydy2019Interpolating}, but the computational expense of these approaches can still become significant for small levels of entropic regularization.  This makes such approaches intractable on quantitative applications where accurate approximations of the unregularized optimal coupling are required.  For example, when the application demands maintaining the fluid dynamics interpretation of optimal transport \cite{Benamou2000Computational}, as it does in the sea ice velocity estimation problem of \cite{Parno2019Remote}. In addition, regularized formulations result in diffuse couplings that also cause blurring in image processing applications like barycentric interpolation \cite{Benamou2015Iterative}. A direct deblurring method like total variation reuglarizations may not recover ideal images \cite{Cuturi2018Semidual}.

The exact (unregularized) solution to OT is currently only feasible for certain subclasses of OT.  OT on discrete measures is fundamentally an assignment problem and can be formulated as linear programming (LP).  If the number of Dirac masses in the discrete measures is not too large, LP can be solved directly.    Semi-discrete OT, where one measure is discrete and the other is continuous, is also naturally cast as an finite dimensional optimization problem (e.g., \cite{Merigot2016Minimal,Kitagawa2019Convergence}) and can often be solved exactly. Continuous OT problems, where both measures admit densities with respect to the Lebesgue measure, admit a PDE formulation based on the Monge-Amp\`ere equation, which can be solved efficiently in 2 or 3 dimensions to obtain the OT solution \cite{Benamou2014Ampere}.   The authors of \cite{Jacobs2020BF} also provide an alternative method for OT with continuous distributions that can be represented on a uniform grid in $\mathbb{R}^d$. Their approach, called the ``back-and-forth method'' (BFM), lays the foundations for our MMOT solver and will be discussed in more detail in \Cref{subsec:gd}.

Similar to entropy regularized OT, there are regularized formulations of MMOT that admit approximate numerical solutions \cite{Benamou2015Iterative, Elvander2020MMOT, Haasler2021Tree}, following with complexity analysis \cite{Lin2022Complexity, Fan2022Complexity}. However, these approaches in general suffer the same numerical instability and blurring issues. Semidefinite relaxation to MMOT is proposed in \cite{Khoo2020Semidefinite} and provides as a lower bound to MMOT. The approach of \cite{Neufeld2022Numerical} provides a regularization-free alternative for approximately solving MMOT problems with controllable levels of sub-optimality.  However, the scalability of this approach to higher dimensional spaces with complex marginals is unclear. Authors of \cite{Altschuler2022Polynomial} also provide a LP-based polynomial-time algorithm to solve some MMOT with structure exactly, and graphical structure is one of them. They use the ellipsoid algorithm with an oracle related with $c$-transform (see \cref{def:c-tran}). The ellipsoid algorithm to LP is known to be slow in practice. Our goal is to construct a fast and exact (to within numerical tolerances) MMOT solver that can scale to marginal distributions derived from high resolution imagery. In practice, our method can be applied on MINST dataset, with more than 30 marginals and much more than 120 gridpoints, comparing the recently proposed GenCol method \cite{Friesecke2022GenCol}. 

\subsection*{Contribution}
In this paper, we develop and analyze a novel algorithm for the efficient solution of continuous MMOT problems with pairwise costs.  More specifically, our method can deal with all cost function in a pairwise form $c(x_1,\cdots,x_m)=\sum_{i<j}c_{ij}(x_i,x_j)$, which includes most classical cost functions used in MMOT except the determinant form in \cite{Carlier2008determinant}. In this category of cost functions, there is a natural graphical structure between marginals that our approach exploits to construct an efficient MMOT solver.  In particular, inspired by the ``back-and-forth'' method of \cite{Jacobs2020BF} for the classic two-marginal OT setting, we derive gradient updates and efficient $c$-transform routines that can be combined to solve the pairwise MMOT problem. The pushforward map, obtained as part of our computed MMOT solution, is accurate and can be utilized in regularizaiton-free applications, including denoising and the Wasserstein barycentric interpolation. 

The paper is organized as follows. \Cref{sec:pre} consists of two parts. In the first part, we provide with all ingredients to understand the back-and-forth method (BFM). The concept of gradient in the Hilbert space is the key to BFM. In the second part, we introduce the MMOT problem with a focus on the duality theory, and basic graph theory for our description.   In \Cref{sec:graph}, we introduce the graphical representation of MMOT under assumptions \ref{A1}--\ref{A3}, and develop the theory necessary to reformulate any MMOT of such type into an equivalent MMOT problem with a tree representation; this is encapsulated in \Cref{thm:unroll}.  \Cref{sec:alg} introduces the main algorithm \Cref{alg:root-tree} to solve any MMOT that has a tree representation. Numerical results with empirical convergence rates studies are presented in \Cref{sec:app}. Extensions to the Wasserstein barycenter problem are described in \Cref{subsec:bary} as an important application of our methods.  We close with concluding thoughts in \Cref{sec:summary}.
\section{Preliminaries}\label{sec:pre}

\subsection{Two-Marginal Primal Formulation}
The classic OT problem is a resource allocation problem. Given two Borel probability measures $\mu_1,\mu_2$ on metric spaces $X_1,X_2$ and a continuous cost function $c: X_1\times X_2\to [0,+\infty]$, the classic Monge OT problem is to find the cheapest way to transport $\mu_1$ to $\mu_2$:
\begin{equation}
\label{MP}
\inf_T\left\{\int c(x,T(x))\mathrm{d}\mu_1(x): \push{T}{\mu_1}=\mu_2\right\}.
\end{equation}
Measures $\mu_1$ and $\mu_2$ are often referred as the \textit{source measure} and the \textit{target measure}, respectively. The \textit{transport map} $T$ satisfies the push-forward condition $\push{T}{\mu_1} = \mu_2$, which is shorthand notation for the condition that $\mu_2(A)=\mu_1(T^{-1}(A))$ for all Borel sets $A\subset X_2$.  The optimal transport map $T^*$ is called the \textit{Monge map}.  The Monge map does not always exist. For example, if $\mu_1$ has fewer atoms than $\mu_2$, then mass from the same point in $X_1$ must be split to multiple points in $X_2$, which cannot be accomplished by any deterministic map $T$. Kantorovich provided a relaxation of \eqref{MP} that circumvents this issue.  The Kantorovich problem takes the form
\begin{equation}
    \label{KP}
    \inf\left\{\int_{X_1\times X_2}c(x_1,x_2)\mathrm{d}P(x_1,x_2): P\in\Gamma(\mu_1,\mu_2)\right\},
\end{equation}
where $\Gamma(\mu_1,\mu_2)$ is the set of \textit{transport plans} defined by
\[\Gamma(\mu_1,\mu_2)=\left\{P\in \P(X_1\times X_2): \push{\pi_1}{P}=\mu_1,~\push{\pi_2}{P}=\mu_2 \right\},\]
and $\pi_1,\pi_2$ are projections on each coordinate\footnote{We distinguish between the projection of measures and the canonical projection of measures. For example, given a probability measure $P$ on the space $(X_1\times X_2)\times X_3\cdots \times X_m$, then the projection of measures $\push{\pi_1}{P}=P_{1}\in\P(X_1), \push{\pi_2}{P}=P_2\in\P(X_2)$, while the canonical projection of measures $\push{\proj_1}{P}=P_{12}\in\P(X_1\times X_2), \push{\proj_2}{P}=P_3\in \P(X_3)$. This will be used in \Cref{lem:gluing}.}. The set of transport plans $\Gamma(\mu_1,\mu_2)$ consists of all joint probability measures on $X_1\times X_2$ with marginals $\mu_1$ and $\mu_2$. If no confusion may arise, we also use $P_i$ as shorthand notation for $\push{\pi_i}{P}$.  When the source distribution $\mu_1$ is atomless, the transport plan $P^\ast$ solving the Kantorovich problem collapses onto the graph of the Monge solution and $P^\ast = \push{\id, T^\ast}{\mu_1}$, where $\id : X_1\rightarrow X_1$ is the identity map (see \cite{Villani2003TOT}).

\subsection{Two-Marginal Duality Theory}\label{subsec:ot-dual}
Instead of solving \eqref{KP} directly, it is often more efficient to solve the dual form.  From here on we will restrict our attention to spaces $X_1=X_2=\Omega\subset (\R^d,\abs{\cdot}_2)$ that are convex and compact, as well as costs in the form $c(x_1,x_2)=h(x_1-x_2)$ for a strictly convex function $h:\R^d\mapsto \R$. Furthermore, we assume all marginals $(\mu_i)$ are probability measures that are absolutely continuous with respect to the Lebesgue measure.   Under these constraints, we have the following theorem.

\begin{theorem}[\cite{Santambrogio2015OTAM}, Theorem 1.40]
The dual problem to \eqref{KP}
\begin{equation}
    \label{DP}
    \sup\left\{\int_{X_1} f_1\mathrm{d}\mu_1 +\int_{X_2} f_2\mathrm{d}\mu_2\,:\, f_1, f_2\in L^1, f_1(x_1)+f_2(x_2)\leqslant c(x_1,x_2)\right\}
\end{equation}
admits a $c$-conjugate solution $(f_1,f_2)$. That is, $f_1(x_1)=\inf_{x_2} c(x_1,x_2)-f_2(x_2)$ and $f_2(x_2)=\inf_{x_1} c(x_1,x_2)-f_1(x_1)$. (See \Cref{def:c-tran}.) Furthermore, the strong duality between \eqref{KP} and \eqref{DP} holds.
\end{theorem}
As a result of strong duality, the maximal objective value in the dual problem \eqref{DP} is equal to the minimum objective in the primal problem \eqref{KP}.  For $c(x_1,x_2)=\abs{x_1-x_2}^2$, the optimal value is called as the \textit{Wasserstein distance} $W^2_2(\mu_1,\mu_2)$. 

As discussed in \cite{Ambrosio2013User}, because the optimal dual variables satisfy $f_1(x_1)+f_2(x_2)=c(x_1,x_2)$ on the support of the optimal coupling $P^*$, the Monge map can be recovered from the optimal dual solution $f_1,f_2$ when the map exists. See \cref{lem:trans} and \cref{thm:trans} in the supplementary document.

\subsection{$c$-Transform}\label{subsec:c-trans}
The constraint in \eqref{DP} induces a key concept in computational OT: the $c$-transform. The $c$-transform is a natural generalization of the more common Legendre-Fenchel transform $f^*(y)=\sup_{x}x \cdot y - f(x)$. 
\begin{definition}[$c$-transform]
\label{def:c-tran}
The $c$-transform of a function $f:X_1\mapsto \R$ is given by
\[f^c(x_2)=\inf_{x_1} c(x_1,x_2)-f(x_1).\]
In addition, we say that $f$ is $c$-concave if there exists a function $g: X_2\mapsto \R$ such that $f=g^c$. We say $(f_1,f_2)$ are $c$-conjugate if $f_1=f_2^c$ and $f_2=f_1^c$.
\end{definition}

\begin{remark}
\label{rmk:ctran}
The $c$-transform cannot decrease the objective value in \eqref{DP}. For any feasible dual variables $f_1$ and $f_2$, and any fixed point $x_2$, we have $f_2(x_2) \leqslant c(x_1,x_2) - f_1(x_1)$ for all $x_1$.  This implies that $f_2(x_2)\leqslant f_1^c(x_2)$ and subsequently $\int_{X_2} f_2 \mathrm{d}\mu_2 \leqslant \int_{X_2} f_1^c \mathrm{d}\mu_2$.
\end{remark}

Given $c(x_1,x_2)=h(x_1-x_2)$ for some strictly convex function $h(\cdot):\R^d\mapsto \R$,
the map $S_f(x_1)\defeq x_1-\nabla h^*(\nabla f(x_1))$ will serve as a key ingredient in the classical OT theory. In particular, if $f=g^c$ for some continuous function $g$, then $S_f(x_1)$ is the unique minimizer to $\inf_{x_2}c(x_1,x_2)-g(x_2)$. Please refer to the supplemental document.

\subsection{Gradient-based Optimization for OT}
\label{subsec:gd}
As a result of \Cref{rmk:ctran} (also see Proposition 1.11 in \cite{Santambrogio2015OTAM}), the dual problem is equivalent to either of the following problems
\begin{subequations}
\label{eq:func}
\begin{align}
    \label{eq:func_I}&\sup\left\{I_1(f_1) = \int_{X_1}f_1\mathrm{d}\mu_1 + \int_{X_2}f_1^c\mathrm{d}\mu_2\,:\, f_1 \textrm{ is } c\textrm{-concave}\right\};\\
    \label{eq:func_J}&\sup\left\{I_2(f_2) = \int_{X_1}f_2^c\mathrm{d}\mu_1 + \int_{X_2}f_2\mathrm{d}\mu_2\,:\, f_2 \textrm{ is } c\textrm{-concave}\right\},
\end{align}
\end{subequations}
whose maximizers are guaranteed to exist.  Concavity and existence of a $c$-concave maximizer was proved by Brenier \cite{Brenier1991Polar} for $c(x_1,x_2)=\frac{1}{2}\abs{x_1-x_2}^2$ and by Gangbo and McCann \cite{Gangbo1996Geometry} for more general cost functions.

The concavity of \eqref{eq:func} suggests that some gradient-based optimization could be effective at solving these problems.  A gradient ascent step, for example, would take the form 
\begin{equation}
f^{(k+1)} = f^{(k)} + \sigma \nabla I(f^{(k)}), 
\label{eq:grad_ascent}
\end{equation}
where $f^{(k)}$ is the value of the dual variable at optimization iteration $k$, $\sigma\in\mathbb{R}$ is a step size parameter, and $\nabla I(f^{(k)})$ is a functional gradient of $I$ with respect to a dual variable $f$ (i.e., $f_1$ in \eqref{eq:func_I} or $f_2$ in \eqref{eq:func_J}).  The Fr\'{e}chet derivatives provide a mechanism for defining the gradient in a suitable Hilbert space.

\begin{definition}[Fr\'{e}chet derivatives and gradient in the Hilbert space]
Given a separable Hilbert space $(\H,\norm{\cdot}_{\H})$ and a functional $E:\H\mapsto \R\cup\{+\infty\}$, we say a bounded linear operator $\delta E_u:\H\mapsto \R$ is the Fr\'{e}chet derivative of $E$ at $u\in\H$ in the direction $v\in\H$ if
\begin{equation*}
    \lim_{\norm{v}_{\H}\to 0}\frac{\abs{E(u+v)-E(u)-\delta E_u(v)}}{\norm{v}_{\H}}=0.
\end{equation*}
The gradient $\nabla_{\H}E(u)\in\H$ is then defined as an element in the Hilbert space that can be used to compute any directional Fr\'{e}chet derivative through an inner product.  More specifically, we say $\nabla_{\H}E(u)$ is the Hilbert space gradient of $E$ at $u$ if
\begin{equation*}
    \langle \nabla_H E(u),v\rangle = \delta E_u(v),\qquad\textrm{for all }v\in\H.
\end{equation*}
\end{definition}
Note that the choice of Hilbert space defines the inner product and thus the form of the gradient.  Jacobs and L\'{e}ger \cite{Jacobs2020BF} show that for $c(x_1,x_2)=\frac{1}{2}\abs{x_1-x_2}^2$, guaranteeing the ascent of \eqref{eq:grad_ascent} requires that the space $\H$ cannot be weaker than 
\begin{equation*}
    \dot{H}^1(\Omega)\defeq\left\{u:\Omega\mapsto\R\,: \int_{\Omega}u\mathrm{d}x=0,\, \int_{\Omega}\abs{\nabla u(x)}^2\mathrm{d}x<\infty\right\},
\end{equation*}
with the inner product $\langle u_1, u_2\rangle_{\dot{H}^1} = \int_{\Omega}\nabla u_1 \cdot \nabla u_2 \mathrm{d}x.$

Given a cost function $c(x_1,x_2)=h(x_1-x_2)$ for some strictly convex function $h(\cdot):\R^d\mapsto \R$, following techniques in \cite{Gangbo1996Geometry, Gangbo1998Multidimensional}, Lemma 3 in \cite{Jacobs2020BF} shows that the choice $\dot{H}^1$ results in the gradients:
\begin{equation}
\label{eq:H1gradient}
\begin{aligned}
    \nabla_{\dot{H}^1}I_1(f_1)&=(-\Delta)^{-1}\left(\mu_1-\push {S_{f_1^c}}{\mu_2}\right);\\    
    \nabla_{\dot{H}^1}I_2(f_2)&=(-\Delta)^{-1}\left(\mu_2-\push {S_{f_2^c}}{\mu_1}\right),
\end{aligned}
\end{equation}
where $\Delta=\nabla\cdot\nabla$ is the Laplacian operator and the pushforward map $S_f(x)$ is given by
\begin{equation}
    \label{eq:pf_map}
    S_{f}(x)\defeq x - \nabla h^*(\nabla f(x)).
\end{equation}
In fact, as shown by Brenier \cite{Brenier1991Polar} for the quadratic cost and by Gangbo and McCann for strictly convex costs, the maximizer $(f_1,f_1^c)$ to \eqref{eq:func_I} (analogously with $(f_2^c,f_2)$ to \eqref{eq:func_J}) induces mappings $S_{f_1}$ and $S_{f_1^c}$ (analogously with $S_{f_2}$ and $S_{f_2^c}$), which satisfy
\begin{itemize}
\item $\push{S_{f_1}}{\mu_1}=\mu_2$ and $\push{S_{f_1^c}}{\mu_2}=\mu_1$;
\item $S_{f_1}: X_1\mapsto X_2$ defines the unique minimizer to $\inf_{x_2} c(x_1,x_2)-f_1^c(x_2)$; and $S_{f_1^c}: X_2\mapsto X_1$ defines the unique minimzer to $\inf_{x_1} c(x_1,x_2)-f_1(x_1)$.
\end{itemize}
In this sense, formulas \eqref{eq:H1gradient} are natural, as one may observe that $\mu_1=\push{S_{f_1^c}}{\mu_2}$ corresponds to the gradient being zero $\nabla_{\dot{H}^1}I_1(f_1)=0\in \dot{H}^{1}$ and the inverse Laplacian $(-\Delta)^{-1}$ stemse from the inner product structure of the Hilbert space $\dot{H}^1$.

\subsection{Multi-Marginal Primal Formulation}
\cite{Pass2015Survey} provided a detailed theoretical survey about MMOT while here we just briefly introduce materials that we will use later. The primal MMOT problem takes the form
\begin{equation}
    \label{eq:KP_MMOT}
\inf_{P\in \Gamma(\mu_1,\cdots,\mu_m)} \int_{\boldsymbol{X}} c(x_1,\cdots,x_m)\mathrm{d}P(x_1,\cdots,x_m),
\end{equation}
for the space $\boldsymbol{X}=X_1\times\cdots\times X_m$ and prescribed marginal probability measures $(\mu_i)_{i=1}^m$.  The set of transport plans $\Gamma(\mu_1,\cdots,\mu_m)$ is defined by
\begin{equation*}
    \Gamma(\mu_1,\cdots,\mu_m)\defeq\left\{P\in \P(\boldsymbol{X}) \mid \push{\pi_i}{P}=\mu_i, 1\leqslant i\leqslant m \right\}.
\end{equation*}
For simplicity, we will also denote the constraint by $P_i=\mu_i$ when the intent is clear. Analogously, the joint marginal $P_{ij}\defeq \push{\pi_{ij}}{P}$ satisfies $\int_{X_1\times\cdots A_i\times\cdots A_j\times\cdots\times X_m}\mathrm{d}P=\int_{A_i\times A_j}\mathrm{d}P_{ij}$ for all Borel sets $A_i\subset X_i$ and $A_j\subset X_j$. 

Cost functions vary in applications of MMOT. In density functional theory, costs of the form $c(x_1,\cdots,x_m)=\sum_{i<j}\abs{x_i-x_j}^{-1}$ or $c(x_1,\cdots,x_m)=\sum_{i<j}-\log\abs{x_i-x_j}$ arise (see \cite{DiMarino2017Repulsive}).  Fluid dynamics use $c(x_1,\cdots,x_m)=\sum_{i=1}^{m-1}\tau^{-1}\abs{x_i-x_{i+1}}^2$ (see \cite{Brenier1989LeastAction}).  Wasserstein barycenters can be formulated as MMOTs with costs of the form $c(x_1,\cdots,x_m)= \sum_{1\leqslant i<j\leqslant m} \lambda_i\lambda_j |x_i-x_j|^2$ (see \cite{Gangbo1998Multidimensional, Agueh2011Barycenters, Trillos2022Multimarginal} and \Cref{subsec:bary}).    An important feature of these costs is that they are all defined pairwise and fall into the form
\begin{equation*}
    c(x_1,\ldots,x_m) = \sum_{1\leqslant i<j\leqslant m} c_{ij}(x_i,x_j).
\end{equation*}
Due to their broad applicability, such pairwise costs will be the focus of this work.

\subsection{Multi-Marginal Duality Theory}
\label{subsec:mmot-dual}

Assume the cost function $c$ is continuous and each $\mu_i$ is supported on a convex and compact subset in $\R^d$. 
The dual problem corresponding to \eqref{eq:KP_MMOT} is given by
\begin{equation}
    \label{eq:DP_MMOT}\sup_{(f_1,\cdots,f_m)} \sum_{i=1}^m \int_{X_i}f_i(x_i)\mathrm{d}\mu_i,
\end{equation}
where $f_i\in L^1(\mu_i)$ and $\sum_{i=1}^m f_i(x_i)\leqslant c(x_1,\cdots,x_m)$. We call the optimal tuple $(f_1,\cdots,f_m)$ the \textit{Kantorovich potentials}. On the opposite, in \Cref{subsec:illustrative} for $m=2$ we call the optimal solution as the optimal loading/unloading prices in particular.

The concept of a $c$-splitting set extends the notion of a $c$-transform to the multi-marginal setting.
A set $G\subset \boldsymbol{X}$ is a \textit{$c$-splitting set}, if there exist $m$ functions $f_i: X_i\mapsto [-\infty,\infty)$ such that
\begin{align*}
\left\{
\begin{aligned}
    &\sum_{i=1}^m f_i(x_i)\leqslant c(x_1,x_2,\cdots,x_m)\qquad \textrm{on}\quad X;\\
    &\sum_{i=1}^m f_i(x_i) = c(x_1,x_2,\cdots,x_m)\qquad \textrm{in}\quad G.
\end{aligned}
\right.
\end{align*}
These functions $(f_1,f_2,\cdots,f_m)$ are called a \textit{$c$-splitting tuple}. 

Much like the $c$-transform was guaranteed to not decrease the dual objective in the two-marginal setting, for a $c$-splitting tuple  $(f_1,\cdots,f_m)$, there exists a \textit{$c$-conjugate tuple} $(\widetilde{f_1},\cdots,\widetilde{f_m})$ that does not decrease the dual objective.  More specifically, for any $i$,
\begin{equation*}
    \widetilde{f_i}(x) =\left(\sum_{j\neq i}\widetilde{f_j}\right)^c\defeq \inf_{\textrm{all~} y_j} c(y_1,\cdots,y_{i-1},x,y_{i+1},\cdots,y_m)-\sum_{j\neq i}\widetilde{f_j}(y_j), 
\end{equation*}
is always a better candidate to \eqref{eq:DP_MMOT}.  This is due to the fact that $f_i(x_i)\leqslant \widetilde{f_i}(x_i)$ for any $x_i\in X_i$. The $c$-conjugate tuple $(\widetilde{f_1},\widetilde{f_2},\cdots,\widetilde{f_m})$ can be constructed from:
\begin{align*}
    \left\{
    \begin{aligned}
    \widetilde{f_1}(x)&=\inf_{\textrm{all~} y_j} c(x,y_2,\cdots,y_m)-\sum_{j\geqslant 2}f_j(y_j);\\
    \widetilde{f_i}(x)&=\inf_{\textrm{all~} y_j} c(y_1,\cdots,y_{i-1},x,y_{i+1},\cdots,y_m)-\sum_{j<i}\widetilde{f_j}(y_j)-\sum_{j>i}f_j(y_j).
    \end{aligned}
    \right.
\end{align*}
This is called a ``convexification trick" in the context of convex analysis \cite{Carlier2008determinant}. 

The existence of optimal solutions to \eqref{eq:KP_MMOT} and \eqref{eq:DP_MMOT}, and the strong duality between these problems, was proved by \cite{Kellerer1984Duality} for more general cost functions. 
When $X_i=\R^d$, \cite{Carlier2008determinant} also showed the existence of a $c$-conjugate optimal solution to \eqref{eq:DP_MMOT} 
for general continuous cost functions.  An extension of this result by \cite{Pass2011Uniqueness} provides an important explicit connection between the primal and dual MMOT problems. Let $(\widetilde{f_1},\cdots, \widetilde{f_m})$ be a $c$-conjugate solution to the dual problem \eqref{eq:DP_MMOT} and let $P$ be the optimal transport plan to primal problem \eqref{eq:KP_MMOT}, then we have the following connection between the primal solution and dual solution
\begin{equation*}
    \sum_{i=1}^m \widetilde{f}_i(x_i)=c(x_1,\cdots,x_m)\qquad P\mathrm{-a.e.}.
\end{equation*}

\subsection{Graph Theory}

A graph $G=(V,E)$ is described by a set of nodes $V$ and a set of edges $E$.  We will consider graphs where each node is associated with a marginal distribution $\mu_i$, or with a dual variable $f_i$, and will simply use the index $i$ to denote the node.  An edge $e=(s,t)\in E$ from node $s$ to $t$ will then represent the pairwise cost $c_{st}(x_s,x_t)$.   We will restrict our attention to \textit{simple graphs}, which do not contain self-loops or multiple edges. 

We will use both \textit{directed} (where the order of the edge $(s,t)$ matters) and \textit{undirected graphs} (where edges are order-agnostic).  For directed graphs, we will use $N^+(i) = \{e_t : e\in E \text{ and } e_s=i\}$ and $N^-(i) : \{e_s : e\in E \text{ and } e_t=i\}$ to denote the sets of downstream nodes and upstream nodes neighboring node $i$, respectively.  A particular type of graph called a \textit{tree} will also be important to our approach.  A tree is a graph where any two nodes are only connected by a single path.  A \textit{rooted tree} is a tree graph where all of the edges flow towards a single node, called the root. In a rooted tree with edges pointing towards root node $r$, the cardinality of the downstream set $|N^+(i)|=1$ for $i\neq r$ and $|N^+(r)|=0$. 

\section{Graphical Representation of MMOT}\label{sec:graph}
In \Cref{sec:alg} we will introduce a novel computational approach to solving MMOT problems with pairwise costs.  A key component of that approach is the graphical representation of pairwise MMOT problems established below.\footnote{We adopt the terminology that the MMOT problem itself admits a graph structure, rather than saying that the \textbf{cost function} has a graph structure, which is the terminology used in \cite{Haasler2021Probabilistic, Haasler2021Tree}.  This distinction prevents confusion from the branched optimal transport problem (see e.g., \cite{Xia2003Path,Maddalena2003Irrigation,Bernot2008Structure}). 
}
This section also establishes one of our main theoretical results in \Cref{thm:unroll}: that any MMOT problem with pairwise cost can be represented as a rooted tree with nodes corresponding to marginal distributions and edges corresponding to pairwise costs.

While the duality theory described in \Cref{subsec:mmot-dual} holds in more general settings, the rest of this paper will consider cost functions that satisfy the following three assumptions:
\begin{enumerate}[label=(A\arabic*)]
    \item \label{A1} The cost function can be expressed as a sum of pairwise costs: \[c(x_1,\cdots,x_m)=\sum_{1\leqslant i<j\leqslant m}c_{ij}(x_i,x_j);\]
    \item \label{A2} At least $(m-1)$ functions $c_{ij}(x_i,x_j)$ are not identically zero;
    \item \label{A3} For each pair $(i,j)$, $c_{ij}(x_i,x_j)=h_{ij}(x_i-x_j)$ for some strictly convex function $h_{ij}$. 
\end{enumerate}
The pairwise assumption \ref{A1} causes the cost function to have a graph structure. The assumption \ref{A2} ensures that the MMOT cannot be divided into multiple independent MMOT problems.  Note that problems that do not satisfy the assumption \ref{A2} can be trivially split into subproblems that do satisfy this assumption.  The assumption \ref{A3} enables a concise representation of the pushforward map (see discussions in \Cref{subsec:gd} or \Cref{thm:trans} in the supplemental document) and enables computationally efficient evaluations of $c_{ij}$-transforms. As a result of the general duality theory established by \cite{Kellerer1984Duality}, under \ref{A1}--\ref{A3} the primal and dual MMOT problems in \eqref{eq:KP_MMOT} and \eqref{eq:DP_MMOT} both admit solutions, and strong duality holds.

The MMOT problem \eqref{eq:KP_MMOT} under assumption \ref{A1} is analogous to an undirected graph with $m$ nodes representing the marginals, and with edges for all costs $c_{st}(x_s,x_t)$ that are not identically $0$.  This bijection is illustrated in \Cref{fig:graph_bijection} for several MMOT problems.

As we show below in \Cref{sec:alg}, MMOT problems that admit tree representations, like \Cref{fig:UG_tree}, can be efficiently solved with a gradient-based optimization strategy generalized from the back-and-forth method of \cite{Jacobs2020BF} for two-marginal OT problems.    In \Cref{thm:unroll} below, we show that the solution of MMOT problems without this tree structure (e.g., \Cref{fig:UG_cycle} and \Cref{fig:UG_bary}) can be obtained through the solution of a larger ``unrolled'' problem that does exhibit this tree structure.  This fundamental result allows us to apply our tree-based approach to any MMOT problem with pairwise costs.

\begin{figure}[htb]
\centering
\begin{subfigure}{.3\linewidth}
    \centering
    \begin{forest}
    for tree={
    edge = {-},
    circle,
    minimum size=5mm,
    inner sep=0pt,
    draw,
    math content,
    tier/.wrap pgfmath arg={tier #1}{level()},
    anchor=center
    },
    [\mu_1 [\mu_2 [\mu_3]] [\mu_4]]
    \end{forest}
    \caption{$c=c_{12}+c_{23}+c_{14}$.}\label{fig:UG_tree}
\end{subfigure}
\begin{subfigure}{.3\linewidth}
    \centering
    \begin{forest}
    for tree={
    edge = {-},
    circle,
    minimum size=5mm,
    inner sep=0pt,
    draw,
    math content,
    tier/.wrap pgfmath arg={tier #1}{level()},
    anchor=center
    },
    [\mu_1
      [\mu_2 
       [,phantom]
       [\mu_3, name=3]
      ]
      [\mu_4, name=4]
    ]
    \draw (4.south) -- (3.north east);
    \end{forest}
    \caption{$c=c_{12}+c_{14}+c_{23}+c_{34}$.}\label{fig:UG_cycle}
\end{subfigure}
\begin{subfigure}{.3\linewidth}
    \centering
    \begin{forest}
    for tree={
    edge = {-},
    circle,
    minimum size=5mm,
    inner sep=0pt,
    draw,
    math content,
    tier/.wrap pgfmath arg={tier #1}{level()},
    anchor=center
    },
    [\mu_1, name=1
      [\mu_2, name=2
       [,phantom]
       [\mu_3, name=3]
      ]
      [\mu_4, name=4]
    ]
    \draw (4.south) -- (3.north east);
    \draw (1.south) -- (3.north);
    \draw (2.east) -- (4.west);
    \end{forest}
    \caption{$c=\sum_{i<j}c_{ij}$.}\label{fig:UG_bary}
\end{subfigure}
\caption{Undirected graphical representations of MMOT problems with $m=4$ marginals.  (a) shows a cost that can be represented directly as a tree, whereas (b) and (c) are cyclic graphs.  \Cref{thm:unroll} provides a mechanism for unrolling these cyclic graphs in equivalent MMOT problems with tree structures.  Also note that (c) shows the structure that arises in the MMOT formulation of Wasserstein barycenter problems.}
\label{fig:graph_bijection}
\end{figure}
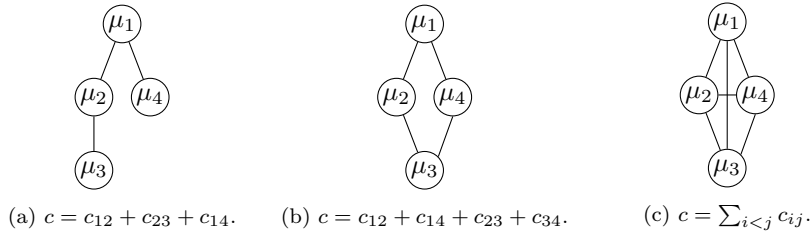

The following lemma establishes a connection between the size of an MMOT problem with cyclic graph and the size of an equivalent problem with a tree structure.
\begin{lemma}\label{lem:duplicate}
Given an undirected graph $G=(V,E)$ with possible cycles, we need exactly $\abs{E}+1-\abs{V}$ duplicate nodes to be unrolled into a tree.
\end{lemma}
\begin{proof}

Adding duplicate nodes and replacing original edges, the new graph retains the same number of edges while becoming a tree. A tree of $\abs{E}$ edges has $\abs{E}+1$ nodes. Thus we need to add $\abs{E}+1-\abs{V}$ duplicate nodes. \end{proof}

To show that solving any pairwise MMOT is equivalent to solving another MMOT with a tree representation, we will need the following generalized gluing lemma.
\begin{lemma}[Generalized gluing lemma, Theorem A.1 in \cite{deAcosta1982Invariance}]
\label{lem:gluing}
Let $J$ be an arbitrary index set and for each $j\in J$, let $X_j, Y_j$ be Polish spaces and $\boldsymbol{X}=\Pi_{j\in J}X_j, \boldsymbol{Y}=\Pi_{j\in J}Y_j$. Let $\phi_j: X_j\mapsto Y_j$ be a measurable map and let $\mu_j\in \P(X_j)$. Let $Q \in \P(\boldsymbol{Y})$ such that $\push{\phi_j}{\mu_j}=\push{\mathbf{Proj}_j}{Q}$ for all $j\in J$.

Then there exists a $P\in\P(\boldsymbol{X})$ such that:
\begin{equation*}
\left\{
\begin{aligned}
    &\push{\mathbf{Proj}_j}{P}=\mu_j,\qquad\textrm{for all }j\in J;\\
    &\left(\left(\phi_j\circ\mathbf{Proj}_j\right)_{j\in J}\right)_{\#}{P}=Q.
\end{aligned}
\right.
\end{equation*}
\end{lemma}

We are now ready to show that any pairwise MMOT problem can be ``unrolled'' into an equivalent MMOT problem with a tree structure.  Qualitatively, the process of unrolling a graph is illustrated in \Cref{fig:unroll}, where a single cycle is broken by duplicating the marginal $\mu_4$.  This process is made more precise in \Cref{thm:unroll}.

\begin{theorem}\label{thm:unroll}
Given a cost function $c(x_1,\cdots,x_m)$ satisfying assumptions \ref{A1}, \ref{A2} and \ref{A3} that corresponds to an undirected graph $G=(V,E)$ with possible cycles, let $n=\abs{E}+1$.
\begin{enumerate}[label=(\alph*)]
\item There exists a map: $\bar{T}:X_1\times\dots\times X_m \mapsto X_1\times \dots\times X_n$ such that for the cost function $\bar{c}(x_1\cdots,x_n)$ on $X_1\times\dots\times X_n$ with $c=\bar{c}\circ \bar{T}$, which corresponds to an undirected tree $\bar{G}=(\bar{V},\bar{E})$ with $\abs{\bar{E}}=\abs{E}=n-1$ and $\abs{\bar{V}}=n$, and we have
\begin{align}
\label{eq:mmot-primal-change}
    \inf_{P^{(m)}\in\Gamma(\mu_1,\cdots,\mu_m)} \int c(x_1,\cdots,x_m)\mathrm{d}P^{(m)}=\inf_{P^{(n)}\in\mathcal{Q}^{(n)}_2} \int\bar{c}(x_1,\cdots,x_n) \mathrm{d}P^{(n)},
\end{align}
where $\mathcal{Q}^{(n)}_2 =\left\{Q=\push{\bar{T}}{P}\in \mathbb{P}(X_1\times\cdots\times X_n)\mid P\in \Gamma(\mu_1,\cdots,\mu_m) \right\}$
\item Let $P^{(m)}$ and $(f_i)_{i=1}^m$ be the optimal primal and dual solutions to the original MMOT \eqref{eq:KP_MMOT} and \eqref{eq:DP_MMOT}, and let $P^{(n)}$ and $(\bar{f}_i)_{i=1}^m$ be the optimal primal and dual solutions to the new MMOT:
\begin{equation}
\label{eq:new-mmot}
\begin{aligned}
    &\inf_{P^{(n)}\in \Gamma(\mu_1,\cdots,\mu_n)}\int \bar{c}(x_1,\cdots,x_n)\mathrm{d}P^{(n)};\\
    &\sup_{\bar{f}_1 +\cdots +\bar{f}_n \leqslant \bar{c}} \sum_{i=1}^n \int \bar{f}_i(x_i)\mathrm{d}\mu_i,
\end{aligned}
\end{equation}
where $(\mu_k)_{k=m+1}^n$ are duplicated from $(\mu_i)_{i=1}^m$ in the unrolling process (shown in \cref{fig:unroll} as an illustration). Then the new MMOT provides a lower bound to the original MMOT, that is:
\begin{equation}
    \label{eq:MMOT_lower}
    \begin{aligned}
    \inf_{P^{(m)}\in \Gamma(\mu_1,\cdots,\mu_m)}\int c\mathrm{d}P^{(m)}=
    &\sup_{f_1 +\cdots + f_m \leqslant c} \sum_{i=1}^m \int f_i(x_i)\mathrm{d}\mu_i\\
        \geqslant &\sup_{\bar{f}_1 +\cdots +\bar{f}_n \leqslant \bar{c}} \sum_{i=1}^n \int \bar{f}_i(x_i)\mathrm{d}\mu_i\\
        =&\inf_{P^{(n)}\in \Gamma(\mu_1,\cdots,\mu_n)}\int \bar{c}\mathrm{d}P^{(n)}.
    \end{aligned}
\end{equation}
\item Assume all $c_{ij}$ and the optimal dual solutions are $C^{2,\alpha}$, then the equality in \eqref{eq:MMOT_lower} holds. That is, solving the original MMOT is equivalent to solve the new MMOT. Furthermore, for any $i$, the original optimal dual solution $f_i$ is the sum of all $\bar{f}_j$ whose nodes are duplicated from $\mu_i$.
\end{enumerate}
\end{theorem}

\begin{figure}[htb!]
\hspace{0.5in}
\begin{minipage}{0.43\textwidth}
\begin{tikzpicture}[scale=0.7]
\centering
\node[circle,draw, scale=0.7] (3) at  (0,0) {$f_3,\mu_3$};
\node[circle,draw, scale=0.7] (1) at  (0,2)  {$f_1,\mu_1$};
\node[circle,draw, scale=0.7] (5) at  (2,0)  {$f_5,\mu_5$};
\node[circle,draw, scale=0.7] (2) at  (-2,0)  {$f_2,\mu_2$};
\node[circle,draw, scale=0.7] (4) at  (-2,-2)  {$f_4,\boldsymbol{\mu_4}$};
\draw (3) -- (1)--(5)--(3);
\draw (1) -- (2)--(4)--(3);
\end{tikzpicture}
\end{minipage}
\begin{minipage}{0.43\textwidth}
\begin{tikzpicture}[scale=0.7]
\node[circle,draw, scale=0.7] (3bar) at  (0,0) {$\widetilde{f}_3,\mu_3$};
\node[circle,draw, scale=0.7] (1bar) at  (0,2)  {$\widetilde{f}_1,\mu_1$};
\node[circle,draw, scale=0.7] (5bar) at  (2,0)  {$\widetilde{f}_5,\mu_5$};
\node[circle,draw, scale=0.7] (2bar) at  (-2,0)  {$\widetilde{f}_2,\mu_2$};
\node[circle,draw, scale=0.7] (4bar) at  (-2,-2)  {$\widetilde{f}_4,\boldsymbol{\mu_4}$};
\node[circle,draw, scale=0.7] (6bar) at  (0,-2)  {$\widetilde{f}_6,\boldsymbol{\mu_4}$};
\draw (3bar) -- (1bar) -- (5bar) -- (3bar) -- (6bar);
\draw (3bar) -- (1bar) -- (2bar) -- (4bar);
\end{tikzpicture}
\end{minipage}
\caption{Breaking one cycle in the left graph results in the right graph with an additional node.   \Cref{thm:unroll} repeats this process to obtain an equivalent MMOT with a tree structure.\label{fig:unroll}}
\end{figure}
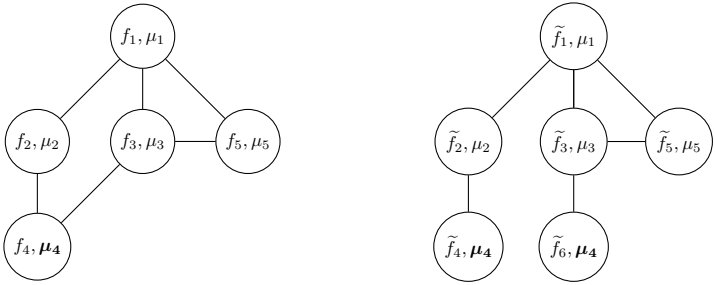

\begin{proof}
We prove it by iteration. Assume there is a cycle in $G=(V,E)$ and we can remove an edge $e=(i,j)$ to break the cycle. Let $c(x_1,\cdots,x_m)=c_{ij}(x_i,x_j)+d(x_1,\cdots,x_m)$ and $\widetilde{c}(x_1,\cdots,x_{m+1})=c_{ij}(x_i,x_{m+1})+d(x_1,\cdots,x_m)$. By defining a map: $T:X_1\times\cdots\times X_m\mapsto X_1\times \cdots\times X_{m}\times X_{m+1}$ with $X_{m+1}=X_j$, given by $T(x_1,\cdots,x_j,\cdots,x_m)=(x_1,\cdots,x_j,\cdots,x_m,x_j)$, we note that: $c=\widetilde{c}\circ T$, thus for any $P\in \P(X_1\times\cdots\times X_m)$, we have
\begin{equation}
\label{eq:composite-ot}
    \int c(x_1,\cdots,x_m)\mathrm{d}P = \int \widetilde{c}\circ T \mathrm{d}P = \int \widetilde{c}\mathrm{d}[\push{T}{P}].
\end{equation}
Note that this unrolling process preserves the number of edges, thus repeat this process and one can obtain the map $\bar{T}$ and the cost function $\bar{c}(x_1,\cdots,x_n)$ that corresponds to an undirected tree $\bar{G}=(\bar{V},\bar{E})$ with $\abs{\bar{E}}=n-1$ and $\abs{\bar{V}}=n$. By change of variables, we prove the part (a):
\begin{equation*}
  \inf_{P^{(m)}\in\Gamma(\mu_1,\cdots,\mu_m)} \int c(x_1,\cdots,x_m)\mathrm{d}P^{(m)}=\inf_{P^{(n)}\in\mathcal{Q}^{(n)}_2} \int\bar{c}(x_1,\cdots,x_n) \mathrm{d}P^{(n)}.
\end{equation*}

Now, we go back to unroll for one step and prove part (b) by iteration. Consider the following two subsets $\mathcal{Q}_1, \mathcal{Q}_2$ of couplings in $\P(X_1\times\cdots\times X_{m}\times X_{j})$, where
\begin{equation*}
\begin{aligned}
    &\mathcal{Q}_1=\left\{Q  \mid Q_i=\mu_i\textrm{ for all }i\notin\{j,m+1\}; Q_{j,m+1}=\push{(\id,\id)}{\mu_j}\right\},\\
    &\mathcal{Q}_2=\left\{Q=\push{T}{P}\in \P(X_1\times\cdots\times X_{m}\times X_{j})\mid P\in\Gamma(\mu_1,\cdots,\mu_m)\right\}.
\end{aligned}
\end{equation*}
Recall $Q_i$ is the shorthand notation for $\push{\pi_i}{Q}$, analogously with $Q_{j,m+1}$.

We first show that these two sets are equivalent. 

For a coupling $Q\in \mathcal{Q}_2$, when $i\notin\{j,m+1\}$ and for any Borel sets $A_i$ in $X_i$, we have 
\begin{equation*}
    Q_i(A_i)=\push{\pi_i}{Q}(A_i)=\push{\pi_i}{(\push{T}{P})}(A_i)=P(T^{-1}\circ \pi_i^{-1}(A_i))=\mu_i(A_i).
\end{equation*}
Also, for any Borel sets $A_j$ in $X_j$ and $A_{m+1}$ in $X_{m+1}=X_j$:
\begin{align*}
Q_{j,m+1}(A_j\times A_{m+1})&=\push{(\pi_j,\pi_{m+1})}{Q}(A_j\times A_{m+1})\\
&=P(T^{-1}\circ (\pi_j,\pi_{m+1})^{-1} (A_j\times A_{m+1}))\\
&=P(T^{-1}(X_1\times\cdots \times A_j\times X_{j+1}\times\cdots \times X_m\times A_{m+1}))\\
&=P(X_1\times\cdots \times (A_j \cap A_{m+1})\times\cdots \times X_{m})\\
&=\mu_j(A_j\cap A_{m+1})=\push{(\id,\id)}{\mu_j}(A_j\times A_{m+1}).
\end{align*}
Thus $Q\in\mathcal{Q}_1$, showing that $\mathcal{Q}_2\subseteq\mathcal{Q}_1$.

On the other hand, consider a coupling $Q\in\mathcal{Q}_1$.   Comparing with \Cref{lem:gluing}, we have $(X_i,\mu_i)$, take $Y_i=X_i$ for $i\neq j$ and $Y_j=X_j\times X_{m+1}$ with $X_{m+1}=X_j$, thus $\boldsymbol{Y}=X_1\times \cdots \times(X_j\times X_{m+1})\times \cdots X_m$. For $i\neq j$, define $\phi_i=\id:X_i\mapsto Y_i$ and $\phi_j:X_j\mapsto Y_j$ by $\phi_j=(\id,\id)$. Up to a permutation, we treat $Q\in\P(X_1\times\cdots\times X_m\times X_j)$ as $Q\in \P(\boldsymbol{Y})$, thus by the definition of $\mathcal{Q}_1$, we have:
\begin{equation*}
    \push{\phi_i}{\mu_i}=\push{\proj_i}{Q};\qquad\textrm{and}\qquad
    \push{(\id,\id)}{\mu_j}=\push{\proj_j}{Q}
\end{equation*}
Thanks to \Cref{lem:gluing}, there exists a $P\in\P(\boldsymbol{X})$ such that:
\begin{equation*}
\left\{
\begin{aligned}
    &\push{\proj_i}{P}=\mu_i\qquad\textrm{for all }i;\\
    &\push{(\phi_1\circ \proj_1,\cdots,\phi_j\circ \proj_j,\cdots,\phi_m\circ \proj_m)}{P}=Q.
\end{aligned}
\right.
\end{equation*}
Note that
\begin{equation*}
\begin{aligned}
    &(\phi_1\circ \proj_1,\cdots,\phi_j\circ \proj_j,\cdots,\phi_m\circ \proj_m)(x_1,\cdots,x_j,\cdots,x_m)\\
    =&(\phi_1(x_1),\cdots,\phi_j(x_j),\cdots,\phi_m(x_m))=(x_1,\cdots,(x_j,x_j),\cdots,x_m),
\end{aligned}
\end{equation*}
which is a permutation of $T$. As we prove the existence of $P\in \Gamma(\mu_1,\cdots,\mu_m)$, thus $Q\in\mathcal{Q}_2$ and subsequently $\mathcal{Q}_1\subseteq \mathcal{Q}_2$.  Combined with the discussion above, this implies that $\mathcal{Q}_1=\mathcal{Q}_2$.

Using this equivalence and \eqref{eq:composite-ot} results in 
\begin{equation*}
\begin{aligned}
    \inf_{P}\int c(x_1,\cdots,x_m)\mathrm{d}P
    =\inf_{Q\in \mathcal{Q}_2}\int \widetilde{c}(x_1,\cdots,x_{m+1})\mathrm{d}Q
    =\inf_{Q\in\mathcal{Q}_1}\int \widetilde{c}(x_1,\cdots,x_{m+1})\mathrm{d}Q.
    \end{aligned}
\end{equation*}

On one hand, by \Cref{subsec:mmot-dual} we have the strong duality for the original MMOT under the cost $c(x_1,\cdots,x_m)$:
\begin{equation}
    \label{eq:duality-ori-ot}
    \inf_{P\in\Gamma(\mu_1,\cdots,\mu_m)}\int c(x_1,\cdots,x_m)\mathrm{d}P= \sup_{f_1+\cdots+f_m\leqslant c}\sum_{i=1}^m \int f_i(x_i)\mathrm{d}\mu_i.
\end{equation}
On the other hand, we also have the strong duality for the new MMOT under the cost $\widetilde{c}(x_1,\cdots,x_{m+1})$ for $\mu_{m+1}=\mu_j$:
\begin{equation}
    \label{eq:duality-new-ot}
\inf_{Q\in\Gamma(\mu_1,\cdots,\mu_m,\mu_{m+1})}\int \widetilde{c}(x_1,\cdots,x_{m+1})\mathrm{d}Q= \sup_{\widetilde{f}_1+\cdots+\widetilde{f}_{m+1}\leqslant \widetilde{c}}\sum_{i=1}^{m+1} \int \widetilde{f}_i(x_i)\mathrm{d}\mu_i.
\end{equation}
Due to the fact that $\mathcal{Q}_1\subset \Gamma(\mu_1,\cdots,\mu_m,\mu_j)$, we have:
\[\inf_{P}\int c\mathrm{d}P=\inf_{Q\in\mathcal{Q}_1}\int \widetilde{c}(x_1,\cdots,x_{m+1})\mathrm{d}Q \geqslant \inf_{Q\in\Gamma(\mu_1,\cdots,\mu_m,\mu_j)}\int \widetilde{c}(x_1,\cdots,x_{m+1})\mathrm{d}Q.\]
By iterating the unrolling from $m+1$ nodes to $n$ nodes, we prove the part (b).

The characteristic function of the set $\mathcal{Q}_1$ is given by:
\begin{align*}
&\sup_{\widetilde{f}_i} \sum_{i\not\in\left\{j,m+1\right\}} \int_{X_i}\widetilde{f}_i\mathrm{d}\mu_i - \int_{X_1\times \cdots \times X_m\times X_{m+1}}\left(\sum_{i\not\in \left\{j,m+1\right\}}\widetilde{f}_i(x_i)\right)\mathrm{d}Q \\
&+ \sup_{g}\int_{X_j\times X_{m+1}} g(x_j,x_{m+1})\mathrm{d}[\push{(\mathrm{id},\mathrm{id})}{\mu_j}] -\int_{X_1\times \cdots\times X_m \times X_{m+1}}g(x_j,x_{m+1})\mathrm{d}Q\\
=&\left\{
\begin{aligned}
&0\qquad & Q\in\mathcal{Q}_1;\\
&+\infty\qquad & Q\in \mathcal{M}_+(X_1\times\dots\times X_m\times X_{m+1})\backslash \mathcal{Q}_1.
\end{aligned}\right.
\end{align*}
among all bounded and continuous functions $\widetilde{f}_i \in C_b(X_i)$ and $g\in C_b(X_j\times X_{m+1})$ for $X_{m+1}=X_j$. By noting that 
\begin{align*}
\int_{X_j\times X_{m+1}} g(x_j,x_{m+1})\mathrm{d}[\push{(\mathrm{id},\mathrm{id})}{\mu_j}]=\int_{X_j}g(x_j,x_j)\mathrm{d}\mu_j,
\end{align*}
we write the following into the Lagrangian form:
\begin{align}
\label{eq:od_dual}
    &\inf_{Q\in\mathcal{Q}_1} \int \widetilde{c}(x_1,\cdots,x_{m+1})\mathrm{d}Q\notag\\
    =&\inf_{Q} \int \widetilde{c}\mathrm{d}Q + \sup_{\widetilde{f}_i, g} \sum_{i\not\in\left\{j,m+1\right\}}\int \widetilde{f}_i\mathrm{d}\mu_i +\int g\mathrm{d}[\push{(\mathrm{id},\mathrm{id})}{\mu_j}]  - \int (g+\sum_{i\not\in \left\{j,m+1\right\}} \widetilde{f}_i )\mathrm{d}Q\notag\\
    =&\left\{
    \begin{aligned}
    &\sup_{\widetilde{f}_i, g} \sum_{i\not\in \left\{j,m+1\right\}}\int_{X_i}\widetilde{f}_i\mathrm{d}\mu_i + \int_{X_j} g(x_j,x_j)\mathrm{d}\mu_j\\
    &\textrm{s.t~} \sum_{i\not\in\left\{j,m+1\right\}} \widetilde{f}_i(x_i)+g(x_j,x_{m+1})\leqslant \widetilde{c}(x_1,\cdots,x_{m+1})
    \end{aligned}\right.
\end{align}

When $g(x_j,x_{m+1})$ is a additively separable function, that is, there exists bounded continuous function $\widetilde{f}_j, \widetilde{f}_{m+1}$ such that $g(x_j,x_{m+1})=\widetilde{f}_j(x_j)+\widetilde{f}_{m+1}(x_{m+1})$, then \eqref{eq:od_dual} turns to be the new MMOT \eqref{eq:duality-new-ot}. Once again, we see that the new MMOT is a lower bound.

For the optimal $Q\in\mathcal{Q}_1$ and the optimal dual solution $(g,\widetilde{f}_1,\cdots,\widetilde{f}_m)$, we have:
\[\sum_{i\not\in \left\{j,m+1\right\}} \widetilde{f}_i(x_i)+g(x_j,x_{m+1})=\widetilde{c}(x_1,\cdots,x_{m+1})=d(x_1,\cdots,x_m)+c_{ij}(x_i,x_{m+1}),\]
$Q$-almost everywhere. That is:
\[g(x_j,x_{m+1})=d(x_1,\cdots,x_m)+c_{ij}(x_i,x_{m+1})-\sum_{i\not\in \left\{j,m+1\right\}} \widetilde{f}_i(x_i)\qquad\textrm{a.e.~} Q.\]
Due to the additional assumption on the differentiability in the part (c), we note that the mixed partial derivatives $\frac{\partial^2}{\partial x_j\partial x_{m+1}}$ are zero on the right hand side, thus at the optimality, $g(x_j,x_{m+1})$ is additively separable. Therefore, by \eqref{eq:duality-ori-ot}, \eqref{eq:duality-new-ot} and \eqref{eq:od_dual}, we have:
\begin{equation}
\label{eq:primal_eq}
\begin{aligned}
    \inf_{P}\int c\mathrm{d}P&=\inf_{Q\in\mathcal{Q}_1}\int \widetilde{c}(x_1,\cdots,x_{m+1})\mathrm{d}Q\\
    &=\sup_{g+\widetilde{f}_1+\cdots+\widetilde{f}_m\leqslant \widetilde{c}(x_1,\cdots,x_{m+1})}\sum_{i\not\in \left\{j,m+1\right\}}\int_{X_i}\widetilde{f}_i\mathrm{d}\mu_i + \int_{X_j} g(x_j,x_j)\mathrm{d}\mu_j\\
    &=\sup_{\widetilde{f}_1+\cdots+\widetilde{f}_{m+1}\leqslant \widetilde{c}}\sum_{i=1}^{m+1}\int_{X_i}\widetilde{f}_i\mathrm{d}\mu_i\\
    &=\inf_{Q\in\Gamma(\mu_1,\cdots,\mu_m,\mu_j)}\int \widetilde{c}(x_1,\cdots,x_{m+1})\mathrm{d}Q.
\end{aligned}
\end{equation}
which sets up the equivalence between the original MMOT and the new MMOT. 

Furthermore, let $(f_i)_{i=1}^m$ and $(\widetilde{f}_i)_{i=1}^{m+1}$ be the Kantorovich potentials. Due to \eqref{eq:duality-ori-ot}, \eqref{eq:duality-new-ot} and \eqref{eq:primal_eq}, we have
\begin{equation*}
\sum_{i=1}^m \int_{X_i}f_i\mathrm{d}\mu_i
=\sum_{i=1}^{m+1}\int_{X_i}\widetilde{f}_i\mathrm{d}\mu_i=\sum_{i\notin\left\{ j,m+1\right\}}\int_{X_i}\widetilde{f}_i\mathrm{d}\mu_i + \int_{X_j}(\widetilde{f}_j(x_j) + \widetilde{f}_{m+1}(x_j))\mathrm{d}\mu_j.
\end{equation*}
Therefore we can define an optimal dual solution $(f_i)_{i=1}^m$ to the original problem, in terms of the optimal dual solution to the new problem:
\begin{equation*}
\left\{
\begin{aligned}
    &f_i=\widetilde{f}_i,\qquad &i\neq j;\\
    &f_j=\widetilde{f}_j+\widetilde{f}_{m+1}\qquad &i=j.
\end{aligned}
\right.
\end{equation*}
The above process can be repeated to remove all cycles in the tree.  \Cref{lem:duplicate} guarantees that only a finite number or repetitions are required, thus completing the proof.
\end{proof}
\begin{remark}
As the regularity theory of the Kantorovich potentials is in general subtle (see Section 6.2 in \cite{Ambrosio2008GF} for example), we impose a strong assumption on the potentials directly in the part (c), in order to obtain the equivalence. For the cost $c_{ij}(x_i,x_j)=\abs{x_i-x_j}^2$, by the Caffarelli's regularity theory, the Kantorovich potentials are $C^{2,\alpha}$ and thus the additional assumption in the part (c) is relieved.  We would like to seek for necessary conditions or weaker sufficient conditions on the cost function for the equivalence in the future work.
\end{remark}

\section{Computational Approach}
\label{sec:alg}

To solve the MMOT problem, we need to maximize the dual functional
\begin{equation}
\label{eq:MMOT_functional}
    I(f_1,\cdots, f_m) = \sum_{i=1}^m \int f_i \mathrm{d}\mu_i
\end{equation}
among dual variables satisfying $\sum_{i=1}^m f_i(x_i)\leqslant c(x_1,\cdots,x_m)$.  Similar to the two-marginal approach of \cite{Jacobs2020BF}, by leveraging $c$-transform to get rid of the constraint, we will use gradient ascent on the remaining $(m-1)$ dual variables in the space $\dot{H}^1$.  As shown below, the graphical interpretation of the MMOT problem will enable fast $c$-transforms and gradient updates.  

On a high level, our algorithm consists of three steps:
\begin{enumerate}[label =\Roman*)]
    \item We first construct an undirected graph with possible cycles based on the cost function.
    \item We follow \Cref{thm:unroll} and ``unroll'' the cyclic graph into an undirected tree, at the cost of adding duplicate nodes.
    \item We solve the unrolled problem with the gradient ascent steps described in \Cref{sec:alg}.
\end{enumerate}

We have discussed Steps I and II in \Cref{sec:graph}. As illustrated in \Cref{fig:direct_graph}, by picking an arbitrary node as the root node and traversing the undirected tree with a breadth first search to add directionality, we obtain a directed rooted tree. Our primary computational task in this section is then finding the solution to MMOT problems with tree representations.

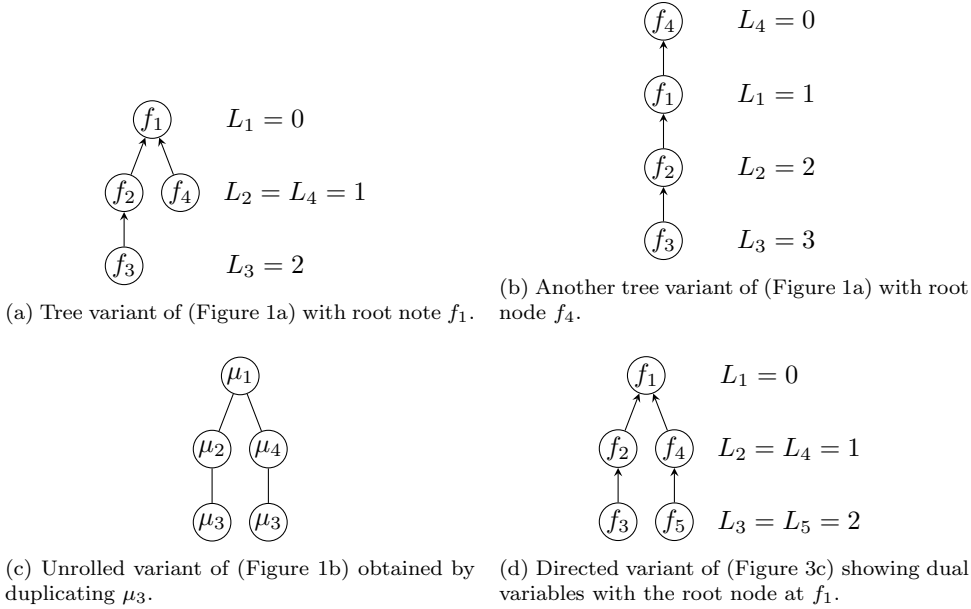
\begin{figure}[htb]
\centering
    \hfill
\begin{subfigure}{.48\linewidth}
    \centering
    \begin{forest}
    for tree={
    edge = {stealth-},
    circle,
    minimum size=5mm,
    inner sep=0pt,
    draw,
    math content,
    tier/.wrap pgfmath arg={tier #1}{level()},
    anchor=center
    },
    [f_1,name=L0, [f_2,name=L1,[f_3,name=L2]], [f_4]]
    \path let \p1 = (L0) in node  at (1.5,\y1) {$L_1=0$};
    \path let \p1 = (L1) in node  at (1.9,\y1) {$L_2=L_4=1$};
    \path let \p1 = (L2) in node  at (1.5,\y1) {$L_3=2$};
    \end{forest}
    \caption{Tree variant of (\Cref{fig:UG_tree}) with root note $f_1$.}\label{fig:DG_tree1}
\end{subfigure}
   \hfill
\begin{subfigure}{.48\linewidth}
    \centering
    \begin{forest}
    for tree={
    edge = {stealth-},
    circle,
    minimum size=5mm,
    inner sep=0pt,
    draw,
    math content,
    tier/.wrap pgfmath arg={tier #1}{level()},
    anchor=center
    },
    [f_4,name=L0,
        [f_1, name=L1, [f_2,name=L2, [f_3, name=L3]]
        ]
    ]
    \path let \p1 = (L0) in node  at (1.5,\y1) {$L_4=0$};
        \path let \p1 = (L1) in node  at (1.5,\y1) {$L_1=1$};
        \path let \p1 = (L2) in node  at (1.5,\y1) {$L_2=2$};
        \path let \p1 = (L3) in node  at (1.5,\y1) {$L_3=3$};
    \end{forest}
    \caption{Another tree variant of (\Cref{fig:UG_tree}) with root node $f_4$.}\label{fig:DG_tree2}
\end{subfigure}

    \hfill
    \begin{subfigure}{.48\linewidth}
        \centering
        \begin{forest}
        for tree={
        edge = {-},
        circle,
        minimum size=5mm,
        inner sep=0pt,
        draw,
        math content,
        tier/.wrap pgfmath arg={tier #1}{level()},
        anchor=center
        },
        [\mu_1
      [\mu_2 
       [\mu_3]
      ]
      [\mu_4, [\mu_3]]
    ]
        \end{forest}
        \caption{Unrolled variant of (\Cref{fig:UG_cycle}) obtained by duplicating $\mu_3$.}\label{fig:UG_unrolled}
    \end{subfigure}
    \hfill
    \begin{subfigure}{.48\linewidth}
        \centering
        \begin{forest}
        for tree={
        edge = {stealth-},
        circle,
        minimum size=5mm,
        inner sep=0pt,
        draw,
        math content,
        tier/.wrap pgfmath arg={tier #1}{level()},
        anchor=center
        },
        [f_1,name=L0, [f_2,[f_3]], [f_4, name=L1, [f_5,name=L2]]]
        \path let \p1 = (L0) in node  at (1.5,\y1) {$L_1=0$};
        \path let \p1 = (L1) in node  at (1.9,\y1) {$L_2=L_4=1$};
        \path let \p1 = (L2) in node  at (1.9,\y1) {$L_3=L_5=2$};
        \end{forest}
        \caption{Directed variant of (\Cref{fig:UG_unrolled}) showing dual variables with the root node at $f_1$.}\label{fig:DG_tree3}
    \end{subfigure}

\caption{Directed tree representations of MMOTs with $m=4$ marginals. We reserve dual variables $(f_i)$ for nodes in directed trees. The first row is two possible directed tree representations of cost in \Cref{fig:UG_tree}. These are constructed by selecting a particular root node (either $f_1$ or $f_4$ in these examples) and then traversing the graph with a breadth first search to add directionality to each edge. The second row demonstrates how our algorithm works from \Cref{fig:UG_cycle}. The layer $L_i$ of node $i$ is needed to provide an ordering in \Cref{alg:root-tree}. \label{fig:direct_graph}}
\end{figure}

\subsection{Illustrative Example}
\label{subsec:illustrative}
To motivate our general gradient ascent approach, first consider a simple MMOT problem with three marginals $\mu_1,\mu_2,\mu_3$ and cost 
\[
c(x_1,x_2,x_3)=c_{12}(x_1,x_2) + c_{23}(x_2,x_3).
\] 
We will need to derive gradients of the dual objective (\eqref{eq:MMOT_functional} with $m=3$) with respect to dual variables.   For this particular example, it can be accomplished by making an analogy between the MMOT problem and multiple two-marginal OT problems.  Due to the gluing lemma, the primal MMOT problem under this cost is analogous to the sum of two OT problems:
\begin{equation*}
    \inf_{P\in \Gamma(\mu_1,\mu_2,\mu_3)} \int c \mathrm{d}P = \inf_{P_{12}\in \Gamma(\mu_1,\mu_2)} \int c_{12} \mathrm{d}P_{12} 
    +\inf_{P_{23}\in \Gamma(\mu_2,\mu_3)} \int c_{23}\mathrm{d}P_{23}.
\end{equation*}
The dual of the MMOT problem is 
\begin{equation*}
\left\{
    \begin{aligned}
\sup_{f_1,f_2,f_3} &  \int f_1(x_1) \mathrm{d}\mu_1 + \int f_2(x_2) \mathrm{d}\mu_2 + \int f_3(x_3) \mathrm{d}\mu_3,\\
\text{s.t.    } & f_1(x_1) + f_2(x_2) + f_3(x_3) \leqslant c(x_1,x_2,x_3),
    \end{aligned}\right.
\end{equation*}
and the dual for the summed two-marginal problems is 
\begin{equation*}
\left\{
    \begin{aligned}
        \sup_{u_1,v_1} & \left[\int u_1(x_1) \mathrm{d}\mu_1 + \int v_1(x_2) \mathrm{d}\mu_2\right] + \sup_{u_2,v_2}\left[\int v_2(x_2) \mathrm{d}\mu_2 + \int u_2(x_3) \mathrm{d}\mu_3\right],\\
        \text{s.t.    } & u_1(x_1) + v_1(x_2) \leqslant c_{12}(x_1,x_2);\\
        & u_2(x_3) + v_2(x_2) \leqslant c_{23}(x_2,x_3),
    \end{aligned}
\right.
\end{equation*}
where $u_1,v_1$ are loading/unloading prices for the OT problem under cost $c_{12}$ and $u_2,v_2$ are loading/unloading prices for the OT problem under cost $c_{23}$.  

\subsubsection{Using $f_2 = (f_1+f_3)^c$}\label{subsec:illustrative1}
In both dual problems, the constraints can be accounted for by using $c$-transforms to define one dual variable in terms of the others.   Assume $f_2(x_2) = (f_1+f_3)^c(x_2)$, $v_1(x_2) = u_1^{c_{12}}(x_2)$, and $v_2(x_2)=u_2^{c_{23}}(x_2)$,  then the dual objective $I_2(f_1,f_3)$ is
\begin{equation*}
    I_2(f_1,f_3) = \int f_1(x_1) \mathrm{d}\mu_1 + \int (f_1+f_3)^c(x_2) \mathrm{d}\mu_2 + \int f_3(x_3) \mathrm{d}\mu_3,
\end{equation*}
and the combined two-marginal dual problems become 
\begin{equation}
\label{eq:combined_dual}
    \sup_{u_1}\left[\int u_1(x_1) \mathrm{d}\mu_1 + \int u_1^{c_{12}}(x_2) \mathrm{d}\mu_2\right] + \sup_{u_2}\left[\int u_2^{c_{23}}(x_2) \mathrm{d}\mu_2 + \int u_2(x_3) \mathrm{d}\mu_3\right].   
\end{equation}
These two expressions have exactly the same form because the pairwise structure of the cost yields $(f_1+f_3)^c(x_2) = f_1^{c_{12}}(x_2) + f_3^{c_{23}}(x_2)$, which implies that 
\begin{equation}
\label{eq:expanded_toy_dual}
\sup_{f_1,f_3} \int f_1(x_1) \mathrm{d}\mu_1 + \int f_1^{c_{12}}(x_2) \mathrm{d}\mu_2 + \int f_3^{c_{23}}(x_2) \mathrm{d}\mu_2 + \int f_3(x_3) \mathrm{d}\mu_3.
\end{equation}
Importantly, this implies that the two-marginal ascent directions in \eqref{eq:H1gradient} can be adapted to define ascent directions for $f_1$ and $f_2$ in \eqref{eq:expanded_toy_dual}. In particular, let $I_2(f_1,f_3) = \int f_1 \mathrm{d}\mu_1 + \int (f_1+f_3)^c \mathrm{d}\mu_2 + \int f_3 \mathrm{d}\mu_3$ denote the dual MMOT objective with $f_2$ defined through the $c$-transform.  Then the gradients take the form
\begin{equation}
    \label{eq:illustrative_grads1}
    \begin{aligned}
    \nabla_{\dot{H}^1}I_2(f_1;f_3) &= (-\Delta)^{-1}(\mu_1-\push{S_{f_1^{c_{12}}}}{\mu_2});\\
    \nabla_{\dot{H}^1}I_2(f_3;f_1) &= (-\Delta)^{-1}(\mu_3-\push{S_{f_3^{c_{23}}}}{\mu_2}).
    \end{aligned}
\end{equation}

The identical relationship between \eqref{eq:expanded_toy_dual} and \eqref{eq:combined_dual} relied on the separable property $(f_1+f_3)^c(x_2) = f_1^{c_{12}}(x_2) + f_3^{c_{23}}(x_2)$. Using the graphical interpretation of MMOT developed in \Cref{sec:graph}, we will later show that this corresponds to the fact that root node $f_2$ is the $c$-transform of leaf nodes $f_1$ and $f_3$ in a rooted tree, as shown in \Cref{fig:update-leaf}. A generalization of this is also considered in \Cref{lem:root-ctrans}.

\subsubsection{Using $f_3 = (f_1+f_2)^c$}
 More care is needed to make an analogy between the MMOT problem and two-marginal problems for different orderings of the $c$-transform.   For example, consider the dual problem when $f_3$ is defined through the $c$-transform of $f_1+f_2$. In this case, the dual objective $I_3(f_1,f_2)$ takes the form 
\[
    I_3(f_1,f_2)=\int f_1(x_1) \mathrm{d}\mu_1 + \int f_2(x_2) \mathrm{d}\mu_2 + \int (f_1+f_2)^c(x_3) \mathrm{d}\mu_3.
\]
Expanding the $c$-transform is more difficult:
\begin{equation}
    \label{eq:f12_to_3}
\begin{aligned}
    (f_1+f_2)^c(x_3) &= \inf_{x_1,x_2} c_{12}(x_1,x_2) + c_{23}(x_2,x_3) - f_1(x_1) - f_2(x_2)\\
    &= \inf_{x_2} c_{23}(x_2,x_3) - f_2(x_2) + f_1^{c_{12}}(x_2)\\
    &= (f_2 - f_1^{c_{12}})^{c_{23}}(x_3),
\end{aligned}
\end{equation}
but still results in a form that can be compared with \eqref{eq:combined_dual}:
\[
    \begin{aligned}
    \sup_{f_1,f_2} & \int f_1(x_1) \mathrm{d}\mu_1 + \int f_2(x_2) \mathrm{d}\mu_2 + \int (f_2 - f_1^{c_{12}})^{c_{23}}(x_3) \mathrm{d}\mu_3\\
    = \sup_{f_1,f_2} & \int f_1 \mathrm{d}\mu_1 + \int f_1^{c_{12}}(x_2) \mathrm{d}\mu_2 + \int (f_2 - f_1^{c_{12}})(x_2) \mathrm{d}\mu_2 + \int (f_2 - f_1^{c_{12}})^{c_{23}} \mathrm{d}\mu_3.
    \end{aligned}
\]
Using the same analogy with \eqref{eq:H1gradient} as above, the gradients of $I_3$ take the form
\begin{subequations}
    \label{eq:illustrative_grads2}
    \begin{align}
        \nabla_{\dot{H}^1}I_3(f_1;f_2) &= (-\Delta)^{-1}(\mu_1-\push{S_{f_1^\prime}}{\mu_2});    \label{eq:illustrative_grads2a}\\
        \nabla_{\dot{H}^1}I_3(f_2;f_1) &= (-\Delta)^{-1}(\mu_2-\push{S_{f^\prime _2}}{\mu_3}),    \label{eq:illustrative_grads2b}
    \end{align}
\end{subequations}
where $f^\prime_1 = (f_1)^{c_{12}}$, $f^\prime_2 = (f_2 - f_1^{c_{12}})^{c_{23}}$ and $S_{f}(x)$ is defined in \eqref{eq:pf_map}.  These identities are made more rigorous in \cref{lem:first_var} in the supplementary document.

The need to include $f_1^{c_{12}}$ in the definition of $f_2^\prime$ stems from the fact that there is no direct pairwise cost relating $x_1$ and $x_3$. The dual variable $f_3$ and $f_1$ are therefore only indirectly coupled through $f_2$, which is illustrated in \Cref{fig:update-mid}.  This is in contrast to \Cref{subsec:illustrative1}, where the root node $f_2$ was directly coupled with $f_1$ and $f_3$.   As we will show in \eqref{eq:net_potential} below, expressions similar to $f_2^\prime$ can be used to propagate information through pairwise MMOT problems with an arbitrary number of marginal distributions.

\begin{figure}[htb]
\centering
\begin{subfigure}{.24\linewidth}
    \centering
    \begin{forest}
    for tree={
      edge = {-},
      circle,
      minimum size=5mm,
      inner sep=0pt,
      draw,
      math content,
      tier/.wrap pgfmath arg={tier #1}{level()},
      anchor=center
        },
  [f_2
  [f_1]
  [f_3]
  ]
\end{forest}
    \caption{Undirected tree for $c=c_{12}+c_{23}$}\label{fig:update-undirected}
\end{subfigure}
\hfill
\begin{subfigure}{.24\linewidth}
    \centering
    \begin{forest}
    for tree={
      edge = {stealth-},
      circle,
      minimum size=5mm,
      inner sep=0pt,
      draw,
      math content,
      tier/.wrap pgfmath arg={tier #1}{level()},
      anchor=center
        },
  [f_2
  [f_1]
  [f_3]
  ]
\end{forest}
    \caption{A directed tree for $c=c_{12}+c_{23}$}\label{fig:update-leaf}
\end{subfigure}
    \hfill
\begin{subfigure}{.24\linewidth}
    \centering
    \begin{forest}
    for tree={
      edge = {stealth-},
      circle,
      minimum size=5mm,
      inner sep=0pt,
      draw,
      math content,
      tier/.wrap pgfmath arg={tier #1}{level()},
      anchor=center
        },
  [f_3
    [f_2 [f_1]
    ]
  ]
  ]
\end{forest}
    \caption{Another directed tree for $c=c_{12}+c_{23}$}\label{fig:update-mid}
\end{subfigure}
    \hfill
\begin{subfigure}{.24\linewidth}
      \centering
      \begin{forest}
      for tree={
          edge = {stealth-},
          circle,
          minimum size=7mm,
          inner sep=0pt,
          draw,
          math content,
          tier/.wrap pgfmath arg={tier #1}{level()},
          anchor=center
          },
      [
      [f_i
        [f_{j_1}]
        [f_{j_2}]
        [f_{j_k}]
      ]
      ]
\end{forest}
 \caption{Example updating a downstream node of leaf nodes}\label{fig:update-nonleaf}
\end{subfigure}
\caption{Following the illustrative example in \Cref{subsec:illustrative}, directed trees play a key role in defining ascent directions and computing $c$-transforms.  As shown by (b) and (c) however, there are multiple directed variants of any undirected tree.}  
\end{figure}
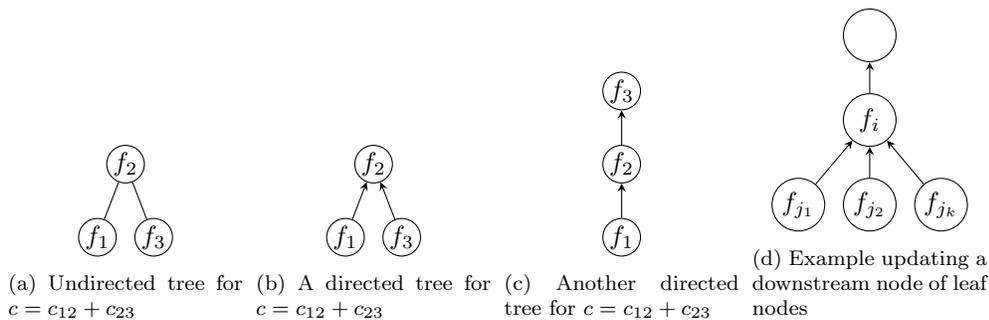

\subsection{Graphical Interpretation and General Dual Gradients}\label{subsec:general-grad}

The undirected tree in \Cref{fig:update-undirected} represents the simple three-marginal problem considered above.  Directed versions of this tree can be defined by choosing a single root node and ensuring that all edges in the tree point towards the root node.  This is shown in \Cref{fig:update-leaf} for root node $f_2$ or in \Cref{fig:update-mid} for root node $f_3$. For either of these choices, the dual variable of the root node is given by the $c$-transform. The gradient in \eqref{eq:illustrative_grads2b} has a slightly different form from \eqref{eq:illustrative_grads1} or \eqref{eq:illustrative_grads2a} because the pushforward map $S_f$ from marginal $\mu_2$ to $\mu_3$ is no longer induced by the $c$-transform of the dual variable $f_2$ purely. Instead, it is induced by $f'_2=(f_2-f_1^{c_{12}})^{c_{23}}$, which we refer to as a \textrm{net potential}. On one hand, for the optimal solution $(f_1,f_2,f_3)$, one may expect $\mu_2=\push{S_{f_3}}{\mu_3}$, and by \eqref{eq:f12_to_3}, we see how the net potential $f'_2$ is constructed. On the other hand, nodes with incoming edges will require using a net potential. Before applying the $c_{23}$-transform, a new potential $f_2^{\textrm{new}}=f_2-f_1^{c_{12}}$ is needed to account for upstream information.  \cref{lem:f_vs_u} in the supplementary document provides a detailed discussion. Loosely speaking, unlike the two-marginal OT, the dual variables in the MMOT problem are no longer purely loading/unloading prices. 

The gradients in \eqref{eq:illustrative_grads1} and \eqref{eq:illustrative_grads2} were obtained by comparing the MMOT dual problem to the sum of dual problems for independent two-marginal OT problems.  This same process can also be employed for larger problems with an arbitrary number of marginal distributions so long as the MMOT cost admits a pairwise cost as in \ref{A1}. 

Consider a directed tree with root node $r$.  Defining $f_r$ through the $c$-transform in \eqref{eq:MMOT_functional} results in a dual functional
\begin{equation*}
    I_r(f_1,\ldots, f_{r-1}, f_{r+1}, \ldots, f_m)\defeq I(f_1,\ldots, \,f_{r-1}, (\sum_{i\neq r} f_i)^c, \,f_{r+1}, \ldots, f_m).
\end{equation*}
In this more general setting and $i\neq r$, the gradient of $I_r$ with respect to $f_i$ takes the form 
\begin{equation}
    \nabla_{\dot{H}^1}I_r(f_i) = (-\Delta)^{-1}\left(\mu_i-\push{S_{f^\prime_i}}{\mu_{N^{+}(i)}}\right),
    \label{eq:general_grad}
\end{equation}
where the net potential $f^\prime_i$ at edge $(i,N^+(i))$ is recursively defined by 
\begin{equation}
    f_i^\prime = (f_i - \sum_{j\in N^{-}(i)}f_j^\prime)^{c_{iN^+(i)}},
    \label{eq:net_potential}
\end{equation}
which is the difference between the dual variable $f_i$ at node $i$ and the sum of upstream net potentials $(f'_j)$. \Cref{fig:update-nonleaf} illustrates the idea.  If node $i$ is a leaf node, the set of upstream nodes is empty $N^{-}(i)=\emptyset$ and the net potential is simply $f_i^\prime = (f_i)^{c_{iN^+(i)}}$.  

\subsection{Gradient Ascent}\label{subsec:grad-ascent}

The gradients defined by \eqref{eq:general_grad} provide a way to update each individual dual variable using gradient ascent while holding the other dual variables fixed.  This can be used to define a block coordinate ascent algorithm for the dual MMOT problem.  At iteration $k$ of the gradient ascent algorithm, the dual variable at node $i$ is updated using
\begin{equation*}
    f_i^{k+1} = f_i^{k} - \sigma \Delta^{-1}\left[\mu_i - \push{S_{f^\prime_i}}{\mu_{N^{+}(i)}} \right],
\end{equation*}
for a step size $\sigma \in \mathbb{R}$. As described in the previous section however, the dual variable at the root node is given by the $c$-transform $f_r = (\sum_{i\neq r} f_i)^c$.  The following lemma provides a mechanism for efficiently computing this $c$-transform using the same net potentials used to define those gradients.

\begin{lemma}\label{lem:root-ctrans}
For a root node $r$ and its upstream nodes $N^-(r)$, we have:
\begin{equation}
\label{eq:root_ctrans}
    f_r(x_r)= \sum_{i\in N^-(r)}f'_i(x_r).
\end{equation}
\end{lemma}
\begin{proof}
When the rooted tree only consists of two layers, the root node and the leaf nodes. By $f_r = (\sum_{i\neq r} f_i)^c$ and the definition \eqref{eq:net_potential}, we have
\begin{align*}
    f_r(x_r)&=\inf_{\textrm{all~} y_i} c(y_1,\cdots,x_r,\cdots,y_m)-\sum_{i\in N^-(r)}f_i(y_i)\\
    &=\inf_{\textrm{all~}y_i}\left[ \sum_{i\in N^-(r)}(c_{ir}(y_i,x_r)-f_i(y_i))\right]=\sum_{i\in N^-(r)}(\inf_{y_i} c_{ir}(y_i,x_r)-f_i(y_i))\\
    &=\sum_{i\in N^-(r)}f_i^{c_{ir}}(x_r)=\sum_{i\in N^-(r)}f'_i(x_r).
\end{align*}

When the rooted tree consists of more than two layers, we may first re-arrange
\begin{equation*}
\begin{aligned}
    &f_r(x_r)=\inf_{\textrm{all~} y_i} c(y_1,\cdots,x_r,\cdots,y_m)-\sum_{i\not = r}f_i(y_i)\\
    &=\inf_{\textrm{all~}y_i}\sum_{i\in N^-(r)}[ c_{ir}(y_i,x_r)-f_i(y_i)-\sum_{j\in\textrm{Tree}(i)\atop j\not= i}f_j(y_j) +\sum_{(j,k)\in\textrm{Tree}(i)}c_{jk}(y_j,y_k)]\\
    &=\sum_{i\in N^-(r)}[\inf_{y_i} \{c_{ir}(y_i,x_r) - f_i(y_i)+\inf_{\textrm{all~}y_j\atop j\in \textrm{Tree}(i)}(\sum_{(j,k)\in\textrm{Tree}(i)}c_{jk}(y_j,y_k)-\sum_{j\in \textrm{Tree}(i) \atop j\not =i}f_j(y_j))\}]
\end{aligned}
\end{equation*}
where we denote a rooted tree with root node $r$ by $\textrm{Tree}(r)=(V,E)$. For simplicity, we slightly abuse notations: $e\in \textrm{Tree}(r)$ ($v\in \textrm{Tree}(r)$) means that an edge (a vertex) belongs to the tree with root node $r$. We can continue this work by re-arranging the infimum by subtrees, to get a nested infimum.

From the inside to the outside of the nested infimum, by noting \eqref{eq:net_potential} and recursively defining $f^{\textrm{new}}_i=f_i-\sum_{j\in N^-(i)}f'_j$ from the leaf nodes towards the root, we obtain \eqref{eq:root_ctrans}.
\end{proof}

Combinining the gradient steps in \eqref{eq:general_grad} with the root node $c$-transform in \eqref{eq:root_ctrans}, results in a method for taking a single gradient ascent step on each dual variable;  this is summarized in \Cref{alg:root-tree}.    As shown in \cite{Jacobs2020BF},  pairwise $c$-transforms can be computed efficiently using the fast Legendre transform when the marginals are discretized on a uniform grid  (see e.g., \cite{lucet1997faster}).   

\begin{algorithm2e}[htb!]
             \SetAlgoLined
             \SetKwFunction{proc}{AscentStep}
             \caption{Gradient ascent step on a rooted tree. \label{alg:root-tree}}
             \SetKwProg{myproc}{Function}{}{}
             
  \myproc{\proc{$(V,E)$, $\{f_1,\ldots,f_m\}$, $\{\mu_1,\ldots,\mu_m\}$, $r$, $\sigma$}}{
          \vspace{1mm}
          \KwData{A tree $(V,E)$ with $m$ nodes; the index $r$ of the root node, potentials $\{f_1,\ldots,f_m\}$ and measures $\{\mu_1,\ldots,\mu_m\}$ at each node; and a stepsize $\sigma$.}
            \vspace{0.1cm}
            \KwResult{Updated values of $\{f_1,\ldots,f_{m}\}$.} 
            \vspace{0.2cm}
            \tcc{Use a breadth-first search to compute the layer $L_{i}$ of node $i$.}
            $L_1,\ldots,L_m = \texttt{BFS}(V,E,r)$\;\vspace{0.2cm}
            \tcc{Find a run order $k_1,\ldots,k_m$ such that $L_{k_s}\geq L_{k_t}$ for $s<t$.}
            $[k_1,\ldots, k_m] \gets$ \texttt{reverse}( \texttt{argsort}$([L_1,\ldots,L_m])\,)$\;\vspace{0.2cm}
                \vspace{0.2cm}
            \tcc{Loop over nodes in graph.}
            \For{$i\gets1$ to $m-1$}{
                \vspace{0.1cm}
                \tcc{Update net potential.}
                $f^\prime_{k_i} \gets \left(f_{k_i} - \underset{j\in N^{-}(k_i)}{\sum} f^\prime_{j}\right)^c$ \;
                \vspace{0.1cm}
                \tcc{Take gradient step.}
                $f_{k_i} \gets f_{k_i} - \sigma\Delta^{-1}\left[\mu_{k_i}-\left(S_{f^\prime_{k_i}}\right)_{\#}\mu_{N^+(k_i)}\right]$\;
            }
            \vspace{0.2cm}
            \tcc{Set root potential to ensure potentials are admissible}
            $f_{k_m} \gets \underset{j\in N^-(k_m)}{\sum} f^\prime_{j}$\;
             \vspace{0.2cm}
             \nl \KwRet $\{f_1,\ldots,f_m\}$\;}
\end{algorithm2e}

To construct a gradient-based optimization scheme, we combine the gradient ascent direction computed by \Cref{alg:root-tree} with a backtracking Armijo line search to choose the step size in a steepest ascent optimization algorithm.  Unlike a standard steepest ascent algorithm however, we have the flexibility at each iteration to change which root node is used to compute the dual gradient and enforce the dual problem constraints.  We can either use a fixed root note or cycle through through all of the possible root nodes. In the two marginal case, \cite{Jacobs2020BF} showed that cycling can help accelerate convergence by keeping the Hessian of the dual problem well-conditioned.  Our empirical results in \Cref{sec:app} indicate that cycling the root node is also critical for fast convergence in the MMOT setting for some test cases. 
\section{Numerical Results}
\label{sec:app}
We now study the performance of our MMOT solver through several numerical examples.   A public GitHub repository with a python implementation of the approach described in \Cref{sec:alg} and all results discussed below can be found in \cite{mmot_github}.   Note that our implementation leverages the \texttt{C} code released in \cite{bfm_github} for fast evaluation of the $c$-transform.

\subsection{Validation}

In this subsection, $c=\sum_{i=1}^{m-1}\frac{1}{2}\abs{x_i-x_{i+1}}^2$. We start with an 4-marginal example as shown in \cref{fig:cycletest1} when marginals $(\mu_i)$ only differ by a translation. As marginals are normalized to be probability measures, the ground truth of optimal transport cost for each test is 0.12. We list (rounded) averaged test results from picking different root nodes, comparing the result from pick $\mu_1$ as the root in the parentheses.

\begin{table}[htb!]
    \centering
    \begin{tabular}{c|cc|cc}
                  & \multicolumn{2}{c|}{Error $10^{-2}$} & \multicolumn{2}{c}{Error $10^{-4}$}\\
    \hline
    Grid Size     & Iterations & Time (s) & Iterations & Time (s) \\
    \hline
    $256\times 256$ & 9 (7) & 0.41 (0.33) & 70 (60) & 2.36 (1.97)\\
    $512\times 512$ & 9 (7) & 1.74 (1.33) & 114 (70) & 19.91 (13.17)\\
    $1024\times 1024$ & 9 (7) & 8.16 (6.86) & 157 (72) & 118.48 (56.21)
    \end{tabular}
    \caption{Compute MMOT cost to \cref{fig:cycletest1}.}
    \label{tab:valid}
\end{table}

Second, we test another 4-marginal example as shown in \cref{fig:cycletest2}. This time we regard the results via the back-and-forth method as the ground truth. Applying the gluing lemma \Cref{lem:gluing}, the optimal objective value is $\frac{1}{2}\sum_{i=1}^{m-1}W_2^2(\mu_i,\mu_{i+1})$. Note that when $m\geqslant 2$, our algorithm saves storage of dual variables, the number of Laplace transform and $c$-transform per iteration, comparing with applying BFM on each $W_2^2(\mu_i,\mu_{i+1})$. However, both of our method and BFM do not have convergence guarantee, though in practice, most test examples stop in few iterations with high accuracy.

\begin{table}[htb!]
    \centering
    \begin{tabular}{c|cc|cc}
                  & \multicolumn{2}{c|}{Error $10^{-3}$} & \multicolumn{2}{c}{Error $10^{-5}$}\\
    \hline
    Grid Size     & Iterations & Time (s) & Iterations & Time (s) \\
    \hline
    $256\times 256$ & 6 (5) & 0.29 (0.23) & 22 (17) & 0.97 (0.81)\\
    $512\times 512$ & 6 (5) & 1.30 (1.01) & 19 (17) & 3.79 (3.49)\\
    $1024\times 1024$ &  6 (5) & 6.43 (4.70) &  20 (19) &  19.46 (18.13)
    \end{tabular}
    \caption{Compute MMOT cost to \cref{fig:cycletest2}}
    \label{tab:valid2}
\end{table}

\subsection{Root Node Cycling}\label{subsec:cycle}
As mentioned in \Cref{sec:alg}, the choice of root node can vary between optimization iterations.  Here we compare the performance of our approach in two scenarios: (1) the root node is fixed throughout the optimization iterations, and (2) the root node is deterministically cycled by choosing it to be $k~(\textrm{mod~} m)$ at the $k^{th}$ iteration.  In all tests the cost function is given by $c(x_0,x_1,x_2,x_3)=\frac{1}{2}(\abs{x_0-x_1}^2+\abs{x_1-x_2}^2+\abs{x_2-x_3}^2)$. 

The results shown in \Cref{fig:cycle-performance} demonstrate that root node cycling can accelerate the convergence dramatically, especially when the marginal distributions have different supports. Loosely speaking, cycling root nodes helps encourage the dual solution to be $c$-conjugate.

\begin{figure}[htb!]
    
    \begin{subfigure}[b]{\textwidth}
    \begin{center}
    \input{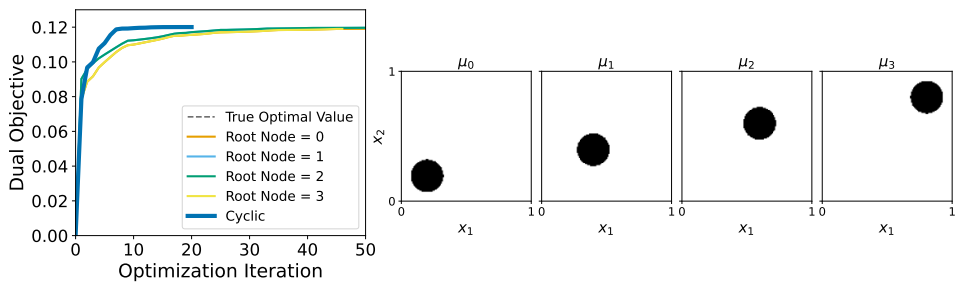}
    \end{center}
    \caption{Impact of cycling the root node with pure translation.\label{fig:cycletest1}}
    \end{subfigure}
    \begin{subfigure}[b]{\textwidth}
    \begin{center}
    \input{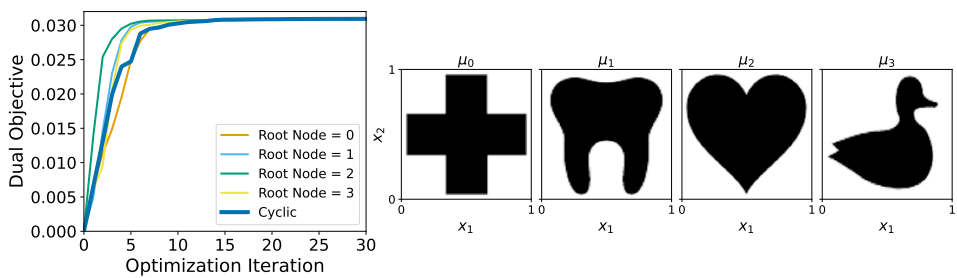}
    \end{center}
    \caption{Impact of cycling the root node with shape deformation.\label{fig:cycletest2}}
    \end{subfigure}
    \begin{subfigure}[b]{\textwidth}
    \begin{center}
    \input{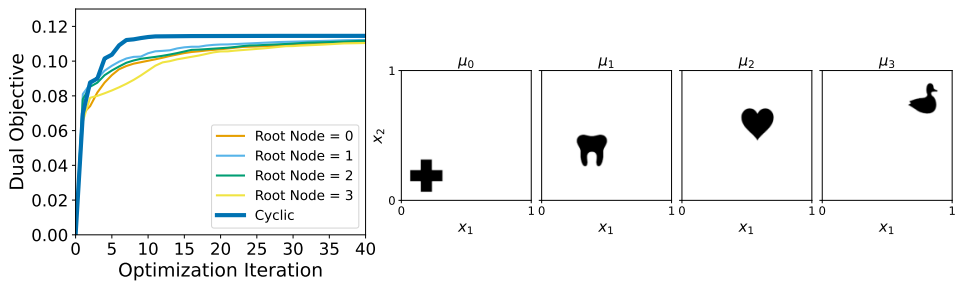}
    \end{center}
    \caption{Impact of cycling the root node with translation and shape deformation.\label{fig:cycletest3}}
    \end{subfigure}
    
    \caption{The impact of cycling through the root node during the gradient step for three different test cases.  In  each four-marginal example, the cost function is given by $c(x_0,x_1,x_2,x_3) = \frac{1}{2}(\abs{x_0-x_1}^2 + \abs{x_1-x_2}^2 + \abs{x_2-x_3}^2)$, which can directly be mapped to a rooted tree without marginal duplication.  The impact of using different directed trees during the gradient step is dramatic in the translation cases where the support of each marginal distribution is distinct.  With root node cycling, the algorithm converges in approximately 10-15 iterations, while the fixed-node gradient approach may not converged to the true value after 250 iterations.}
    \label{fig:cycle-performance}
\end{figure}

\subsection{Wasserstein barycenter}\label{subsec:bary}
Agueh and Carlier \cite{Agueh2011Barycenters} introduced the Wasserstein barycenter problem:
\begin{equation}
    \label{eq:bary}
    \inf_{\mu \in \P(X)} \sum_{i=1}^m \frac{\lambda_i}{2} W_2^2(\mu_i, \mu)
\end{equation}
for a given sequence of probability measures $(\mu_i)\subseteq \P(X)$ and positive weights $(\lambda_i)$. The minimizer $\mu$ is called as the \textit{Wasserstein barycenter}. Without loss of generality,  we assume $\sum_{i=1}^m \lambda_i=1$. Agueh and Carlier showed that \eqref{eq:bary} is equivalent to a MMOT problem under the Gangbo-\'{S}wi\c{e}ch type cost $\displaystyle c(x_1,\cdots,x_m)=\sum_{1\leqslant i<j\leqslant m}\frac{\lambda_i \lambda_j}{2}\abs{x_i-x_j}^2$. Importantly, this cost function includes only pairwise terms and the gradient ascent algorithm described above can also be used.  Once solved, the barycenter $\mu$ can be extracted from any MMOT dual variable $f_i$ with its marginal $\mu_i$:
\begin{equation}\label{eq:root_map}
    \mu = \push{\id-\frac{1}{\lambda_i}\nabla f_i}{\mu_i}.
\end{equation} 
Please refer to the supplementary documents and references there. The pipeline to solve the barycenter problem via our algorithm is illustrated graphically in \Cref{fig:bary}.

\begin{figure}[htb]
\centering
\begin{subfigure}{.3\linewidth}
    \centering
    \includegraphics[width=0.8\linewidth]{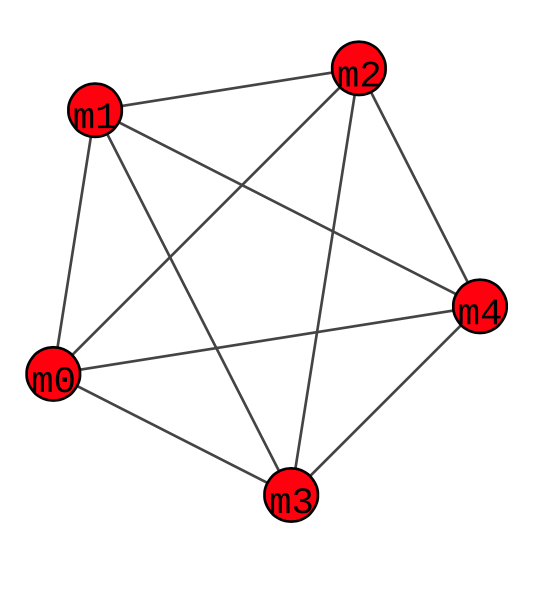}
    \caption{The undirected graph with cycle.\label{fig:bary_cycle}}
\end{subfigure}
    \hfill
\begin{subfigure}{.3\linewidth}
    \centering
    \includegraphics[width=\linewidth]{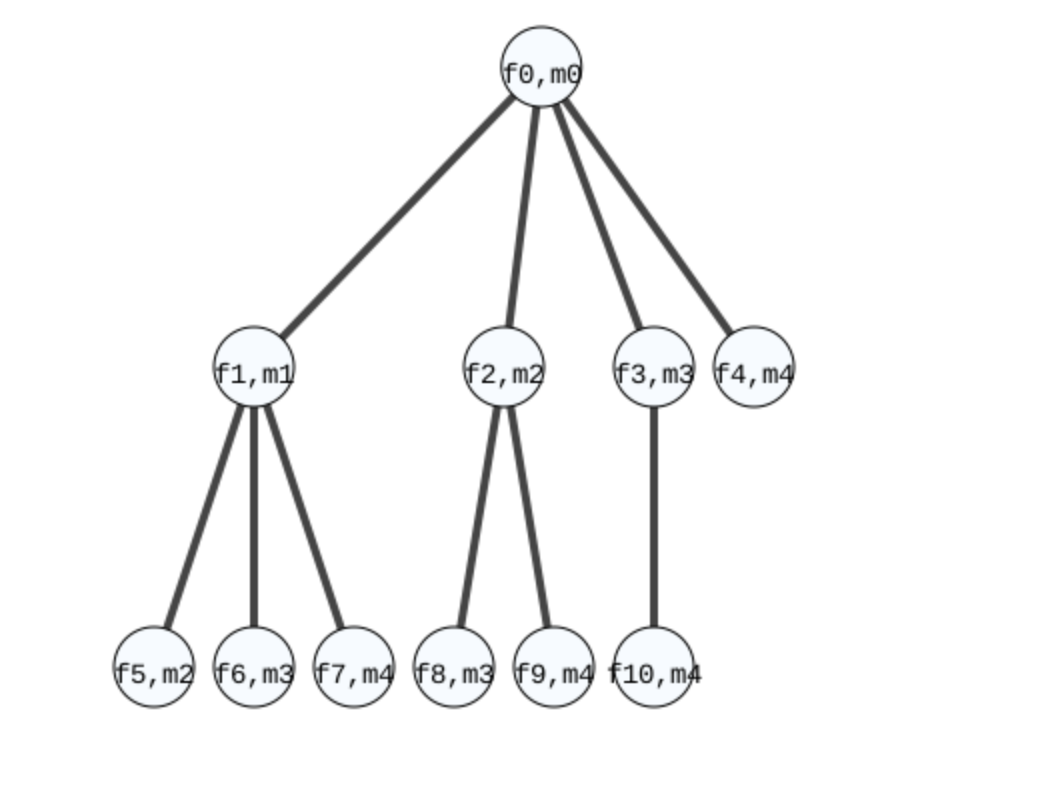}
    \caption{The undirected tree after duplicating nodes.\label{fig:bary_tree}}
\end{subfigure}
   \hfill
\begin{subfigure}{.3\linewidth}
    \centering
    \includegraphics[width=\linewidth]{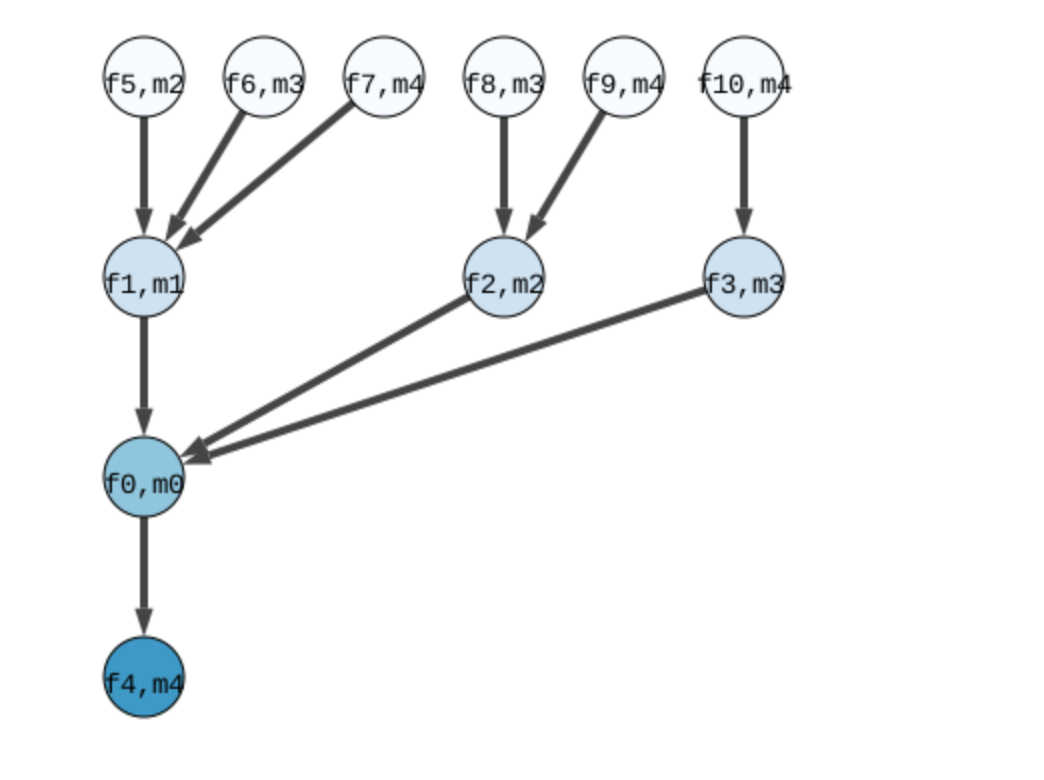}
    \caption{A rooted tree representation of \cref{fig:bary_cycle}.\label{fig:bary_rooted}}
\end{subfigure}
\caption{{The pipeline to compute the Wasserstein barycenter via the MMOT approach. (a) is the corresponding undirected graph representation of the Gangbo-\'{S}wi\c{e}ch type cost. (b) is the undirected graph representation after ``unrolling'' (a) by duplicating the nodes. (c) is the rooted tree representation after picking a root node and is updated by \cref{alg:root-tree}.} \label{fig:bary}}
\end{figure}

First, the Wasserstein barycenter problem is represented as a MMOT with a complete undirected graph representation (see \Cref{fig:bary_cycle}). Second, to solve MMOT under Gangbo-\'{S}wi\c{e}ch type cost, we first unroll this undirected graph by duplicating nodes to remove cycles (see \Cref{fig:bary_tree}).  We then use the method described in \Cref{sec:alg} to solve the unrolled problem, and obtain a dual solution that can be used to compute the barycenter.

\Cref{fig:shape_interpolation} demonstrates the use of this MMOT solution for shape interpolation.  Inspired by an example in the POT (Python Optimal Transport) package \cite{Flamary2021pot}, we use the four marginals ``redcross'', ``heart'', ``tooth'' and ``duck'' shown at the four corners of \Cref{fig:shape_interpolation}.  Each image is $1088\times 1088$ pixels.   

All other plots in  \Cref{fig:shape_interpolation} are a weighted Wasserstein barycenters computing using our MMOT approach.  The weights correspond to bilinear interpolation between the corners.  With comparable computational times to regularized solvers, our method provides much sharper interpolations.

\begin{figure}[htb!]
    \centering
    \includegraphics[width=0.6\textwidth]{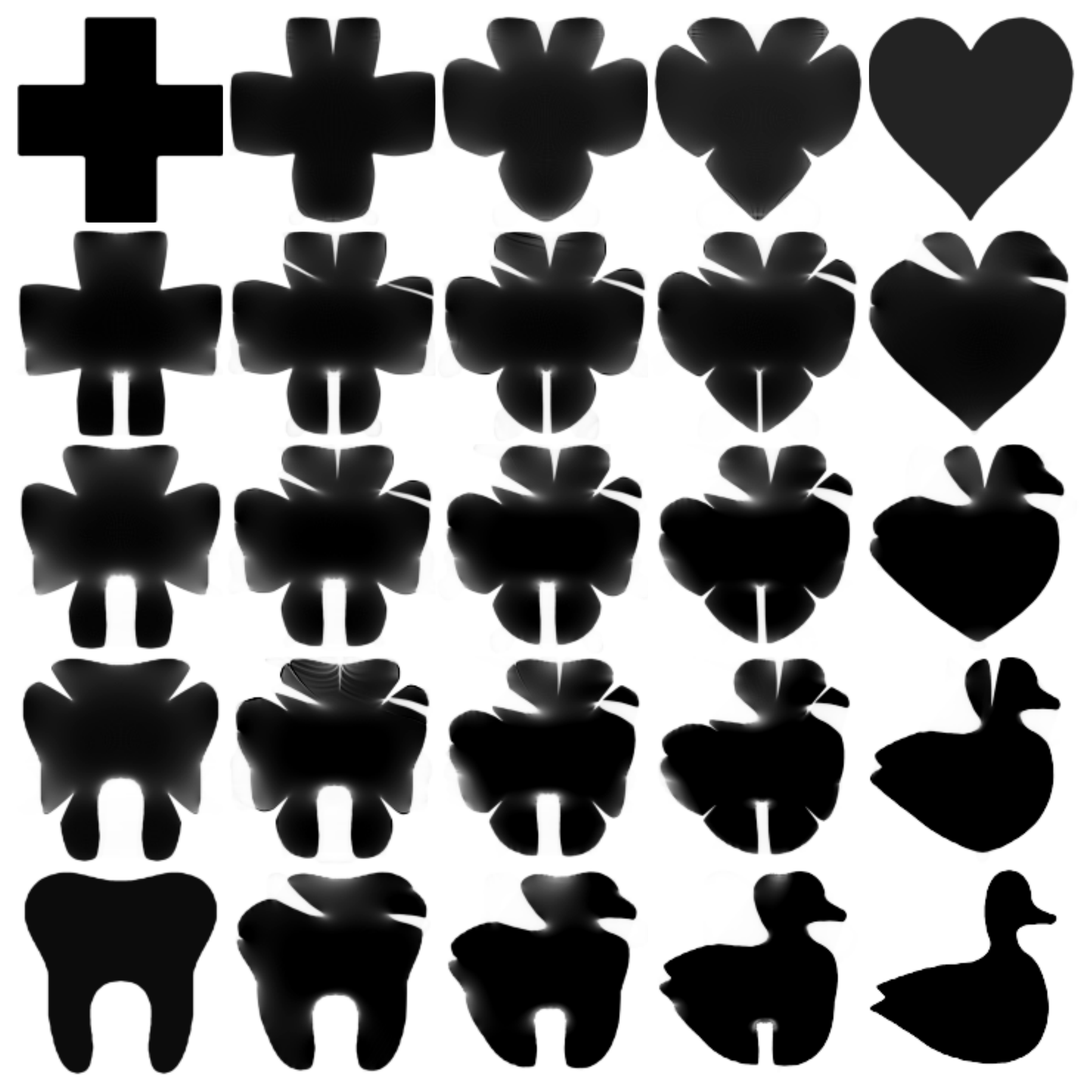}
    \caption{The Wasserstein barycenters of ``redcross'', ``heart'', ``tooth'' and ``duck''. All interpolated images are obtained by solving \eqref{eq:bary} and computing \eqref{eq:root_map}. Those interpolated images show features from four marginals and have negligible diffusion effects.}
    \label{fig:shape_interpolation}
\end{figure}

\section{Summary}\label{sec:summary}

We have presented a novel algorithm for multimarginal optimal transport problems with pairwise cost functions. Our solutions do not require regularizing the MMOT problem and are exact to within solver tolerance.  To the best of our knowledge, this is the first extension of the back-and-forth method (BFM) introduced by \cite{Jacobs2020BF} to the multi-marginal setting and the first approach capable of solving MMOT problems based on high resolution imagery.   We leverage a graphical interpretation of the dual MMOT problem that can be applied to MMOT problems with an arbitrary number of marginals, as long as the cost function admits a pairwise representation.  

As our method is inspired by BFM, our approach has the same gap between theoretical convergence analysis and numerical observations. Finding a convergence result under mild assumptions is therefore an interesting avenue for future work. It is worthy to note the hardness results in \cite{Altschuler2021Hardness}. In the meanwhile, it is also natural to ask if these approaches can be generalized to cost functions which are not the sum of pairwise functions, for example the determinant type of cost function \cite{Carlier2008determinant}. Note that one motivation for pairwise costs is the need for a fast $c$-transform. For pairwise cost function, the $c$-transform in high dimensions can be decomposed into nested 1D $c$-transforms, which can be obtained through fast algorithms of a divide-and-conquer type. As the $c$-transform is crucial to understand classical optimal transport theory, it maybe not be a coincidence that fast $c$-transforms are the key to numerical solutions.

\newpage
\appendix



\section{Supplement to \Cref{subsec:c-trans}}

In this section, we provide with some comparisons between the $c$-transform and the well-known Legendre transform. The Legendre transform not only helps us to understand the $c$-transform, but also helps with our methods in at least two aspects: first, the closed form of optimal transport map for strictly convex function is in terms of the Legendre transform (see \cref{thm:trans}); second, the $c$-transform is done via fast Legendre transform (see \cref{lem:trans}), thanks to the code released in \cite{bfm-github}.

\begin{definition}[Subdifferential] 
The subdifferential $\partial \phi(x)$ is defined as:
\begin{equation*}
    \partial \phi(x)\defeq \left\{y\,\mid x' \cdot y - \phi(x')\mathrm{~is~maximal~ at~} x'=x\right\}.
\end{equation*}
\end{definition}

\begin{definition}[$c$-superdifferential]
The $c$-superdifferential is defined as:
\begin{equation*}
    \partial^c f(x_1)\defeq \left\{x_2\,\mid c(x',x_2)-f(x') \mathrm{~is~minimal~at~} x'=x_1\right\}
\end{equation*}
\end{definition}

\cref{lem:trans} and \cref{thm:trans} compare the Legendre transform with $c$-transform. 
\begin{lemma}[\cite{Santambrogio2015OTAM, Ambrosio2021Lectures}]\label{lem:trans}
For $X_1=X_2=\R^d$,
\begin{enumerate}[label=(\roman*)]
    \item $\phi^{**}\leqslant \phi$, with equality if and only if $\phi$ is convex and lower semi-continuous;
    \item $f^{cc}\geqslant f$, with equality if and only if $f$ is $c$-concave;
    \item For $c(x_1,x_2)=\frac{1}{2}\abs{x_1-x_2}^2$, $f(x_2)$ is $c$-concave if and only if $\phi(x_2)=\frac{1}{2}\abs{x_2}^2-f(x_2)$ is convex and lower semi-continuous. Moreover, 
    $f^c(x_1)=\frac{1}{2}\abs{x_1}^2-\phi^*(x_1).$
    \item For convex function $\phi$ and $\phi^*$, we have
    \begin{equation*}
        y\in\partial\phi(x) \Longleftrightarrow \phi(x)+\phi^*(y)=x\cdot y \Longleftrightarrow x\in\partial \phi^*(y) . 
    \end{equation*}
    \item $x_2\in\partial^c f(x_1)\Longleftrightarrow f(x_1)+f^c(x_2)=c(x_1,x_2) \Longleftrightarrow x_1\in \partial f^c(x_2)$.
\end{enumerate}
\end{lemma}

\begin{theorem}[\cite{Rockafellar1970Convex, Gangbo2004introduction}]\label{thm:trans}
\begin{enumerate}[label=(\roman*)]
\item Given a strictly convex and lower semi-continuous function $\phi(x):\R^d \mapsto \R$, then 
\begin{equation*}
    y = \nabla \phi^*(x)
\end{equation*}
is the unique maximizer to 
\begin{equation*}
    \sup_y \langle x,y\rangle - \phi(y).
\end{equation*}
\item 
Given $c(x_1,x_2)=h(x_1-x_2)$ for some strictly convex function $h$, assume $g(x_2)$ is a compactly supported continuous function, and $f(x_1)=g^c(x_1)$. If $f$ is differentiable at $x_1$, then 
\begin{equation*}
    x_2\defeq x_1- (\nabla h)^{-1}(\nabla f(x_1))= x_1 - \nabla h^*(\nabla f(x_1))
\end{equation*}
is the unique minimizer to 
\[\inf_{x_2} c(x_1,x_2)-g(x_2).\]
That is, $x_2$ is the unique pre-image of $x_1$ under the mapping $\partial^c g$. 
\end{enumerate}
\end{theorem}

\section{Supplementary Lemmas to \Cref{subsec:illustrative}}
In this section, we provide with two supplementary lemmas. \cref{lem:first_var} follows \cite{Gangbo2004introduction} to compute the Fr\'{e}chet derivatives first, in order to define the gradient in $\dot{H}^1$. As the cost function gets more complex, finding the Fr\'{e}chet derivatives through this way can be quite complex. \cref{lem:f_vs_u} serves as one motivation to our algorithm. It shows the difference roles the Kantorovich potentials play, from 2-marginal to multi-marginal. As a result, we are motivated by this to introduce the net potentials $f'_i$ along the rooted tree.
\begin{lemma}\label{lem:first_var}
Let $X_1,X_2,X_3\subset \R^d$ be compact and convex domains and each measure $\mu_i \in \P(X_i)$ has a strictly positive density, and $c(x_1,x_2,x_3)=c_{12}(x_1,x_2)+c_{23}(x_2,x_3)$ where $c_{12}(x_1,x_2)=h_1(x_1-x_2), c_{23}(x_2,x_3)=h_2(x_2-x_3)$ for some  continuously differentiable and strictly convex functions $h_1,h_2$. Define a functional
\[H((f_2+f_3)^c, f_2, f_3)=\int_{X_1} (f_2+f_3)^c(x_1)\mathrm{d}\mu_1 + \int_{X_2}f_2(x_2)\mathrm{d}\mu_2 + \int_{X_3}f_3(x_3)\mathrm{d}\mu_3,\]
over the space of continuous function $f_2:X_2\mapsto \R, f_3:X_3\mapsto \R$. Then 
\[\delta_{f_2}H(f_2;f_3)=\mu_2 - \push{S_{(f_2-f_3^{c_{23}})^{c_{12}}}}{\mu_1}.\]
\end{lemma}
\begin{proof}
The proof follows the proof to lemma 3 in \cite{Jacobs2020BF}. Please refer to Proposition 2.9 in \cite{Gangbo2004introduction} for a detailed proof or refer to \cite{Gangbo1996Geometry} for milder assumptions.

\begin{align*}
    &\lim_{\varepsilon\to 0}\dfrac{H((f_2+\varepsilon\xi+f_3)^c, f_2+\varepsilon\xi, f_3)-H((f_2+f_3)^c,f_2,f_3)}{\varepsilon}\\
    =&\int_{X_1}\frac{(f_2(x_2)+\varepsilon\xi(x_2)+f_3(x_3))^c - (f_2(x_2)+f_3(x_3))^c}{\varepsilon}\mathrm{d}\mu_1 + \int_{X_2} \xi(x_2)\mathrm{d}\mu_2\\
    =&\int_{X_1} \frac{(f_2-f_3^{c_{23}}+\varepsilon\xi)^{c_{12}}-(f_2-f_3^{c_{23}})^{c_{12}}}{\varepsilon}\mathrm{d}\mu_1 + \int_{X_2}\xi(x_2)\mathrm{d}\mu_2\\
    =&-\int_{X_1}\xi(S_{(f_2-f_3^{c_{23}})^{c_{12}}})(x_1)\mathrm{d}\mu_1 + \int_{X_2}\xi (x_2)\mathrm{d}\mu_2\\
    =&-\int \xi \mathrm{d} [\push{S_{(f_2-f_3^{c_{23}})^{c_{12}}}}{\mu_1} + \int \xi\mathrm{d}\mu_2.
\end{align*}
\end{proof}

\begin{lemma}\label{lem:f_vs_u}
Let $X_1,X_2,X_3\subset \R^d$ be compact and convex domains, and each measure $\mu_i\in \P(X_i)$ is absolutely continuous with respect to the Lebesgue measure. For $c(x_1,x_2,x_3)=c_{12}(x_1,x_2)+c_{23}(x_2,x_3)$, we have:
\begin{itemize}
\item If $(u_1,v_1), (u_2,v_2)$ are optimal loading/unloading prices to the OT under cost $c_{12}, c_{23}$ respectively, then $(f_1, f_2, f_3)=(u_1, v_1+u_2, v_2)$ is the Kantorovich potential to the MMOT under the cost $c(x_1,x_2,x_3)$.
    \item If $(f_1,f_2,f_3)$ is the Kantorovich potential to the MMOT under the cost $c$, then $(u_1,v_1)=(f_1, f_1^{c_{12}}), (u_2,v_2)=(f_2-f_1^{c_{12}},f_3)$ are optimal loading/unloading prices to the OT under cost $c_{12}, c_{23}$ respectively.
\end{itemize}
\end{lemma}
\begin{proof}
Given $P\in\P(X_1,X_2,X_3)$, we define $P_{1}(A)=\int_{A\times X_2\times X_3}\mathrm{d}P$ and $P_{1,2}(A\times B)=\int_{A\times B\times X_3}\mathrm{d}P$. On one hand
\begin{align}
    &\inf_{P\in\Gamma(\mu_1,\mu_2,\mu_3)} \int c_{12}(x_1,x_2)+c_{23}(x_2,x_3)\mathrm{d}P\nonumber\\
    =&\inf_{P\in\Gamma(\mu_1,\mu_2,\mu_3)} \int c_{12}(x_1,x_2)\mathrm{d}P_{1,2} +\int c_{23}(x_2,x_3)\mathrm{d}P_{2,3}\nonumber\\
    =&\inf_{Q^1\in\Gamma(\mu_1,\mu_2)} \int c_{12}(x_1,x_2)\mathrm{d}Q^1 +\inf_{Q^2\in\Gamma(\mu_2,\mu_3)}\int c_{23}(x_2,x_3)\mathrm{d}Q^2\nonumber\\
    =&\sup_{u_1+v_1\leqslant c_{12}}\int u_1\mathrm{d}\mu_1 + \int v_1\mathrm{d}\mu_2 + \sup_{u_2+v_2\leqslant c_{23}}\int u_2\mathrm{d}\mu_2 + \int v_2\mathrm{d}\mu_3;\label{eq:uvuv}
\end{align}
On the other hand,
\begin{align}
    \inf_{P\in\Gamma(\mu_1,\mu_2,\mu_3)} \int c_{12}(x_1,x_2)+c_{23}(x_2,x_3)\mathrm{d}P\nonumber\\
    =\sup_{f_1+f_2+f_3\leqslant c}\int f_1\mathrm{d}\mu_1 + \int f_2\mathrm{d}\mu_2 + \int f_3\mathrm{d}\mu_3.\label{eq:f123}
\end{align}
Given a tuple $(u_1,v_1,u_2,v_2)$ that achieves the maximum in \eqref{eq:uvuv}, we define $f_1=u_1, f_2=v_1+u_2, f_3=v_2$, then $$f_1(x_1)+f_2(x_2)+f_3(x_3)=u_1(x_1)+v_1(x_2)+u_2(x_2)+v_2(x_3)\leqslant c_{12}(x_1,x_2)+c_{23}(x_2,x_3)$$
is an admissible solution to \eqref{eq:f123}.
\begin{align*}
    \eqref{eq:uvuv}&=\int u_1\mathrm{d}\mu_1 + \int (v_1+u_2)\mathrm{d}\mu_2 + \int v_2\mathrm{d}\mu_3\\
    &=\int f_1\mathrm{d}\mu_1 + \int f_2\mathrm{d}\mu_2+\int f_3\mathrm{d}\mu_3\\
    &\leqslant \eqref{eq:f123}.
\end{align*}
Since $\eqref{eq:uvuv}=\eqref{eq:f123}$, the tuple $(f_1,f_2,f_3)$ is a maximizer to \eqref{eq:f123}.

Given a tuple $(f_1,f_2,f_3)$ that achieves the maximum in \eqref{eq:f123}, we define $u_1=f_1, v_1=f_1^{c_{12}},u_2=f_2-f_1^{c_{12}},v_2=f_3$.

We first show that $(u_1,v_1,v_2,u_3)$ is an admissible solution to \eqref{eq:uvuv}. By definition, we just need to show that $u_2(x_2)+v_2(x_3)=f_2(x_2)-f_1^{c_{12}}(x_2)+f_3(x_3)\leqslant c_{23}(x_2,x_3)$. By the duality theory, $f_2(x_2)=(f_1+f_3)^{c}=f_1^{c_{12}}(x_2)+f_3^{c_{23}}(x_2)$. Thus
\begin{align*}
    &f_1^{c_{12}}(x_2)+f_3^{c_{23}}(x_2)=f_2(x_2)=v_1(x_2)+u_2(x_2)\\
    \implies &u_2(x_2)=f_3^{c_{23}}(x_2)\\
    \implies &u_2(x_2)=v_2^{c_{23}}(x_2),
\end{align*}
thus $u_2(x_2)+v_2(x_3)\leqslant c_{23}(x_2,x_3)$. As a result, $(u_1,v_1,v_2,u_3)$ is an admissible solution to \eqref{eq:uvuv}. Analogously the above, it is the maximizer to \eqref{eq:uvuv} as well.
\end{proof}

\section{Supplement to \Cref{subsec:bary}}
\begin{theorem}[\cite{Agueh2011Barycenters}]\label{thm:bary}
For any $m$-tuple $(x_1,\cdots,x_m)\in (\R^d)^m$ and weights $(\lambda_1,\cdots,\lambda_m)$ such that $\sum_{i=1}^m \lambda_i=1$, let us define the (Euclidean) barycenter map $T:(\R^d)^m\mapsto \R^d$:
\begin{equation*}
    T(x_1,\cdots,x_m)=\sum_{i=1}^m \lambda_i x_i.
\end{equation*}
The optimal solution $P$ to the MMOT of Gangbo-\'{S}wi\c{e}ch type cost 
\begin{equation}
    \label{eq:GS_MMOT}
    \inf_{P\in\Gamma(\mu_1,\cdots,\mu_m)} \int_{(\R^d)^m} \left(\sum_{1\leqslant i<j\leqslant m}\frac{\lambda_i \lambda_j}{2}\abs{x_i-x_j}^2\right)\mathrm{d}P(x_1,\cdots,x_m)
\end{equation}
induces the barycenter $\mu$ to \eqref{eq:bary} by 
\begin{equation*}
\begin{aligned}
    \mu &= \push{T}{P}= \push{\sum_{j=1}^m \lambda_j T_i^1}{\mu_1};;\notag\\
        &= \push{\id-\frac{1}{\lambda_i}\nabla f_i}{\mu_i}.\notag
\end{aligned}
\end{equation*}
where $(f_i)$ are dual variables to \eqref{eq:GS_MMOT}, and for $x=(x_1,\cdots,x_m)$ $P$-almost everywhere, 
\[x_i = T_i^1(x_1)\defeq\nabla \left(\frac{1}{2}\abs{\cdot}^2 - \frac{f_i}{\lambda_i}\right)^* \circ \nabla\left(\frac{1}{2}\abs{\cdot}^2- \frac{f_1}{\lambda_1}\right)(x_1).\]

\end{theorem}

\begin{proof}
The proof is due to \cite{Gangbo1998Multidimensional}. We follow the discussion in \cite{Agueh2011Barycenters} but in terms of the dual variables $(f_i)$, rather than the variables $g_i(x_i)=\frac{\lambda_i (1-\lambda_i)}{2}\abs{x_i}^2-f_i(x_i)$ used in the convex analysis. More precisely, \cite{Agueh2011Barycenters} consider the primal and dual problems:
\begin{subequations}
    \label{eq:bary_in_convex}
    \begin{align}
        &\sup \int \left(\sum_{1\leqslant i<j \leqslant m} \lambda_i\lambda_j x_i x_j\right)\mathrm{d}P;\\
        &\inf \sum_{i=1}^m \int g_i\mathrm{d}\mu_i\qquad \textrm{~subject to~} \sum_{i=1}^m g_i \geqslant \sum_{1\leqslant i<j\leqslant m}\lambda_i \lambda_j x_i x_j.
    \end{align}
\end{subequations}
We considered the following instead:
\begin{subequations}
    \label{eq:bary_in_dual}
    \begin{align}
        &\inf \int \sum_{1\leqslant i<j \leqslant m} \frac{\lambda_i\lambda_j}{2}\abs{x_i-x_j}^2\mathrm{d}P;\\
        &\sup \sum_{i=1}^m \int f_i\mathrm{d}\mu_i\qquad \textrm{~subject to~} \sum_{i=1}^m f_i \leqslant \sum_{1\leqslant i<j\leqslant m}\frac{\lambda_i \lambda_j}{2}\abs{x_i-x_j}^2.
    \end{align}
\end{subequations}
These two sets of problems are equivalent under the change of variables:
\begin{equation*}
    g_i(x_i) = \frac{\lambda_i (1-\lambda_i)}{2}\abs{x_i}^2 - f_i(x_i).
\end{equation*}
The optimal condition to \eqref{eq:bary_in_convex} is for $P$-a.e. $x=(x_1,\cdots,x_m)$
\begin{align*}
    &\nabla g_i(x_i)=\lambda_i \sum_{j\neq i}\lambda_j x_j\\
    \Longleftrightarrow & \nabla \left(\frac{\lambda_i}{2}\abs{\cdot}^2 + \frac{g_i}{\lambda_i}\right)(x_i)=\sum_{j=1}^m \lambda_j x_j =\nabla \left(\frac{\lambda_1}{2}\abs{\cdot}^2 + \frac{g_1}{\lambda_1}\right)(x_1)\\
    \Longleftrightarrow & x_i = \nabla \left(\frac{\lambda_i }{2}\abs{\cdot}^2+\frac{g_i}{\lambda_i}\right)^*\circ \nabla \left(\frac{\lambda_1}{2}\abs{\cdot}^2+\frac{g_1}{\lambda_1}\right)(x_1)\defeq T_i^1(x_1).
\end{align*}
By the change of variables, the optimal condition to \eqref{eq:bary_in_dual} is for $P$-a.e. $x=(x_1,\cdots,x_m)$
\begin{align*}
    &\nabla f_i(x_i) = \lambda_i (1-\lambda_i)x_i - \lambda_i \sum_{j\neq i}\lambda_j x_j = \lambda_i (x_i - \sum_j \lambda_j x_j)\\
    \Longleftrightarrow &\nabla \left(\frac{1}{2}\abs{\cdot}^2-\frac{f_i}{\lambda_i}\right)(x_i) = \sum_j \lambda_j x_j = \nabla \left(\frac{1}{2}\abs{\cdot}^2-\frac{f_1}{\lambda_1}\right)(x_1)\\
    \Longleftrightarrow &x_i = \nabla \left(\frac{1}{2}\abs{\cdot}^2 - \frac{f_i}{\lambda_i}\right)^* \circ \nabla\left(\frac{1}{2}\abs{\cdot}^2- \frac{f_1}{\lambda_1}\right)(x_1)=T_i^1 (x_1).
\end{align*}
\end{proof}


\section*{Acknowledgments}
We would like to thank Anne Gelb, Yoonsang Lee, Doug Cochran, Xianfeng David Gu and James Ronan for many fruitful conversations.  This work was funded in part by US Office of Naval Research MURI grant N00014-20-1-2595. 

\bibliographystyle{siamplain}
\bibliography{references}

\begin{thebibliography}{10}

\bibitem{Agueh2011Barycenters}
{\sc M.~Agueh and G.~Carlier}, {\em Barycenters in the {W}asserstein space},
  SIAM J. Math. Anal., 43 (2011), pp.~904--924,
  \url{https://doi.org/10.1137/100805741}.

\bibitem{Altschuler2021Hardness}
{\sc J.~Altschuler and E.~Boix-Adserà}, {\em Hardness results for
  multimarginal optimal transport problems}, Discrete Optimization, 42 (2021),
  p.~100669,
  \url{https://doi.org/https://doi.org/10.1016/j.disopt.2021.100669}.

\bibitem{Altschuler2022Polynomial}
{\sc J.~Altschuler and E.~Boix-Adserà}, {\em Polynomial-time algorithms for
  multimarginal optimal transport problems with structure}, Math. Program.,
  (2022), \url{https://doi.org/10.1007/s10107-022-01868-7}.

\bibitem{Ambrosio2021Lectures}
{\sc L.~Ambrosio, E.~Bru\'{e}, and D.~Semola}, {\em Lectures on optimal
  transport}, vol.~130 of Unitext, Springer, Cham, 2021,
  \url{https://doi.org/10.1007/978-3-030-72162-6}.
\newblock La Matematica per il 3+2.

\bibitem{Ambrosio2013User}
{\sc L.~Ambrosio and N.~Gigli}, {\em A user's guide to optimal transport}, in
  Modelling and optimisation of flows on networks, vol.~2062 of Lecture Notes
  in Math., Springer, 2013, pp.~1--155,
  \url{https://doi.org/10.1007/978-3-642-32160-3\_1}.

\bibitem{Ambrosio2008GF}
{\sc L.~Ambrosio, N.~Gigli, and G.~Savar\'{e}}, {\em Gradient flows in metric
  spaces and in the space of probability measures}, Lectures in Mathematics ETH
  Z\"{u}rich, Birkh\"{a}user Verlag, Basel, second~ed., 2008.

\bibitem{Arjovsky2017Wasserstein}
{\sc M.~Arjovsky, S.~Chintala, and L.~Bottou}, {\em Wasserstein generative
  adversarial networks}, in International conference on machine learning, PMLR,
  2017, pp.~214--223.

\bibitem{Benamou2000Computational}
{\sc J.~Benamou and Y.~Brenier}, {\em A computational fluid mechanics solution
  to the {M}onge-{K}antorovich mass transfer problem}, Numer. Math., 84 (2000),
  pp.~375--393, \url{https://doi.org/10.1007/s002110050002}.

\bibitem{Benamou2015Iterative}
{\sc J.~Benamou, G.~Carlier, M.~Cuturi, L.~Nenna, and G.~Peyr\'{e}}, {\em
  Iterative {B}regman projections for regularized transportation problems},
  SIAM J. Sci. Comput., 37 (2015), pp.~A1111--A1138,
  \url{https://doi.org/10.1137/141000439}.

\bibitem{Benamou2014Ampere}
{\sc J.~Benamou, B.~D. Froese, and A.~M. Oberman}, {\em Numerical solution of
  the optimal transportation problem using the {M}onge-{A}mp\`ere equation}, J.
  Comput. Phys., 260 (2014), pp.~107--126,
  \url{https://doi.org/10.1016/j.jcp.2013.12.015}.

\bibitem{Bernot2008Structure}
{\sc M.~Bernot, V.~Caselles, and J.-M. Morel}, {\em The structure of branched
  transportation networks}, Calc. Var. Partial Differential Equations, 32
  (2008), pp.~279--317, \url{https://doi.org/10.1007/s00526-007-0139-0}.

\bibitem{Brenier1989LeastAction}
{\sc Y.~Brenier}, {\em The least action principle and the related concept of
  generalized flows for incompressible perfect fluids}, J. Amer. Math. Soc., 2
  (1989), pp.~225--255, \url{https://doi.org/10.2307/1990977}.

\bibitem{Brenier1991Polar}
{\sc Y.~Brenier}, {\em Polar factorization and monotone rearrangement of
  vector-valued functions}, Comm. Pure Appl. Math., 44 (1991), pp.~375--417,
  \url{https://doi.org/10.1002/cpa.3160440402}.

\bibitem{Brenier2008Generalized}
{\sc Y.~Brenier}, {\em Generalized solutions and hydrostatic approximation of
  the {E}uler equations}, Physica D: Nonlinear Phenomena, 237 (2008),
  pp.~1982--1988.

\bibitem{Carlier2008determinant}
{\sc G.~Carlier and B.~Nazaret}, {\em Optimal transportation for the
  determinant}, ESAIM Control Optim. Calc. Var., 14 (2008), pp.~678--698,
  \url{https://doi.org/10.1051/cocv:2008006}.

\bibitem{Caron2020Unsupervised}
{\sc M.~Caron, I.~Misra, J.~Mairal, P.~Goyal, P.~Bojanowski, and A.~Joulin},
  {\em Unsupervised learning of visual features by contrasting cluster
  assignments}, in Advances in Neural Information Processing Systems, vol.~33,
  2020, pp.~9912--9924.

\bibitem{Courty2017Joint}
{\sc N.~Courty, R.~Flamary, A.~Habrard, and A.~Rakotomamonjy}, {\em Joint
  distribution optimal transportation for domain adaptation}, Advances in
  Neural Information Processing Systems, 30 (2017).

\bibitem{Cuturi2013Sinkhorn}
{\sc M.~Cuturi}, {\em Sinkhorn distances: Lightspeed computation of optimal
  transport}, Advances in neural information processing systems, 26 (2013).

\bibitem{Cuturi2018Semidual}
{\sc M.~Cuturi and G.~Peyr\'{e}}, {\em Semidual regularized optimal transport},
  SIAM Rev., 60 (2018), pp.~941--965, \url{https://doi.org/10.1137/18M1208654}.

\bibitem{De2016Comparing}
{\sc S.~De, A.~P. Bart{\'o}k, G.~Cs{\'a}nyi, and M.~Ceriotti}, {\em Comparing
  molecules and solids across structural and alchemical space}, Physical
  Chemistry Chemical Physics, 18 (2016), pp.~13754--13769,
  \url{https://doi.org/10.1039/C6CP00415F}.

\bibitem{deAcosta1982Invariance}
{\sc A.~de~Acosta}, {\em Invariance principles in probability for triangular
  arrays of {$B$}-valued random vectors and some applications}, Ann. Probab.,
  10 (1982), pp.~346--373,
  \url{http://links.jstor.org/sici?sici=0091-1798(198205)10:2<346:IPIPFT>2.0.CO;2-8&origin=MSN}.

\bibitem{Dessein2018Regularized}
{\sc A.~Dessein, N.~Papadakis, and J.-L. Rouas}, {\em Regularized optimal
  transport and the rot mover's distance}, The Journal of Machine Learning
  Research, 19 (2018), pp.~590--642.

\bibitem{DiMarino2017Repulsive}
{\sc S.~Di~Marino, A.~Gerolin, and L.~Nenna}, {\em Optimal transportation
  theory with repulsive costs}, in Topological optimization and optimal
  transport, vol.~17 of Radon Ser. Comput. Appl. Math., De Gruyter, Berlin,
  2017, pp.~204--256.

\bibitem{Elvander2020MMOT}
{\sc F.~Elvander, I.~Haasler, A.~Jakobsson, and J.~Karlsson}, {\em
  Multi-marginal optimal transport using partial information with applications
  in robust localization and sensor fusion}, Signal Processing, 171 (2020),
  p.~107474.

\bibitem{Fan2022Complexity}
{\sc J.~Fan, I.~Haasler, J.~Karlsson, and Y.~Chen}, {\em On the complexity of
  the optimal transport problem with graph-structured cost}, in Proceedings of
  The 25th International Conference on Artificial Intelligence and Statistics,
  PMLR, 2022, pp.~9147--9165,
  \url{https://proceedings.mlr.press/v151/fan22a.html}.

\bibitem{Feydy2019Interpolating}
{\sc J.~Feydy, T.~S{\'e}journ{\'e}, F.-X. Vialard, S.~Amari, A.~Trouv{\'e}, and
  G.~Peyr{\'e}}, {\em Interpolating between optimal transport and {MMD} using
  {S}inkhorn divergences}, in The 22nd International Conference on Artificial
  Intelligence and Statistics, PMLR, 2019, pp.~2681--2690.

\bibitem{Flamary2021pot}
{\sc R.~Flamary, N.~Courty, A.~Gramfort, M.~Z. Alaya, A.~Boisbunon, S.~Chambon,
  L.~Chapel, A.~Corenflos, K.~Fatras, N.~Fournier, L.~Gautheron, N.~Gayraud,
  H.~Janati, A.~Rakotomamonjy, I.~Redko, A.~Rolet, A.~Schutz, V.~Seguy, D.~J.
  Sutherland, R.~Tavenard, A.~Tong, and T.~Vayer}, {\em {POT}: {P}ython
  {O}ptimal {T}ransport}, Journal of Machine Learning Research, 22 (2021),
  pp.~1--8, \url{http://jmlr.org/papers/v22/20-451.html}.

\bibitem{Friesecke2022GenCol}
{\sc G.~Friesecke, A.~S. Schulz, and D.~V\"{o}gler}, {\em Genetic column
  generation: fast computation of high-dimensional multimarginal optimal
  transport problems}, SIAM J. Sci. Comput., 44 (2022), pp.~A1632--A1654,
  \url{https://doi.org/10.1137/21M140732X},
  \url{https://doi.org/10.1137/21M140732X}.

\bibitem{Frogner2015Learning}
{\sc C.~Frogner, C.~Zhang, H.~Mobahi, M.~Araya, and T.~A. Poggio}, {\em
  Learning with a {W}asserstein loss}, in Advances in neural information
  processing systems, vol.~28, 2015,
  \url{https://proceedings.neurips.cc/paper/2015/file/a9eb812238f753132652ae09963a05e9-Paper.pdf}.

\bibitem{Gangbo2004introduction}
{\sc W.~Gangbo}, {\em An introduction to the mass transportation theory and its
  applications}.
\newblock UCLA lecture notes, 2004,
  \url{https://www.math.ucla.edu/~wgangbo/publications/notecmu.pdf}.

\bibitem{Gangbo1996Geometry}
{\sc W.~Gangbo and R.~J. McCann}, {\em The geometry of optimal transportation},
  Acta Math., 177 (1996), pp.~113--161,
  \url{https://doi.org/10.1007/BF02392620}.

\bibitem{Gangbo1998Multidimensional}
{\sc W.~Gangbo and A.~\'{S}wi\c{e}ch}, {\em Optimal maps for the
  multidimensional {M}onge-{K}antorovich problem}, Comm. Pure Appl. Math., 51
  (1998), pp.~23--45,
  \url{https://doi.org/10.1002/(SICI)1097-0312(199801)51:1<23::AID-CPA2>3.0.CO;2-H}.

\bibitem{Trillos2022Multimarginal}
{\sc N.~Garcia~Trillos, M.~Jacobs, and J.~Kim}, {\em The multimarginal optimal
  transport formulation of adversarial multiclass classification},
  arXiv:2204.12676,  (2022).

\bibitem{Genevay2018Learning}
{\sc A.~Genevay, G.~Peyr{\'e}, and M.~Cuturi}, {\em Learning generative models
  with {S}inkhorn divergences}, in International Conference on Artificial
  Intelligence and Statistics, PMLR, 2018, pp.~1608--1617.

\bibitem{Haasler2021Tree}
{\sc I.~Haasler, A.~Ringh, Y.~Chen, and J.~Karlsson}, {\em Multimarginal
  optimal transport with a tree-structured cost and the {S}chr\"{o}dinger
  bridge problem}, SIAM J. Control Optim., 59 (2021), pp.~2428--2453,
  \url{https://doi.org/10.1137/20M1320195}.

\bibitem{Haasler2021Probabilistic}
{\sc I.~Haasler, R.~Singh, Q.~Zhang, J.~Karlsson, and Y.~Chen}, {\em
  Multi-marginal optimal transport and probabilistic graphical models}, IEEE
  Trans. Inform. Theory, 67 (2021), pp.~4647--4668,
  \url{https://doi.org/10.1109/tit.2021.3077465}.

\bibitem{Jacobs2020BF}
{\sc M.~Jacobs and F.~L\'{e}ger}, {\em A fast approach to optimal transport:
  the back-and-forth method}, Numer. Math., 146 (2020), pp.~513--544,
  \url{https://doi.org/10.1007/s00211-020-01154-8}.

\bibitem{bfm_github}
{\sc M.~Jacobs and F.~L\'{e}ger}, {\em The back-and-forth method}.
\newblock \url{https://github.com/Math-Jacobs/bfm}, 2021.

\bibitem{Kellerer1984Duality}
{\sc H.~G. Kellerer}, {\em Duality theorems for marginal problems}, Z. Wahrsch.
  Verw. Gebiete, 67 (1984), pp.~399--432,
  \url{https://doi.org/10.1007/BF00532047}.

\bibitem{Khoo2020Semidefinite}
{\sc Y.~Khoo, L.~Lin, M.~Lindsey, and L.~Ying}, {\em Semidefinite relaxation of
  multimarginal optimal transport for strictly correlated electrons in second
  quantization}, SIAM J. Sci. Comput., 42 (2020), pp.~B1462--B1489,
  \url{https://doi.org/10.1137/20M1310977},
  \url{https://doi.org/10.1137/20M1310977}.

\bibitem{Kitagawa2019Convergence}
{\sc J.~Kitagawa, Q.~M\'{e}rigot, and B.~Thibert}, {\em Convergence of a
  {N}ewton algorithm for semi-discrete optimal transport}, J. Eur. Math. Soc.
  (JEMS), 21 (2019), pp.~2603--2651, \url{https://doi.org/10.4171/JEMS/889}.

\bibitem{Lin2022Complexity}
{\sc T.~Lin, N.~Ho, M.~Cuturi, and M.~I. Jordan}, {\em On the complexity of
  approximating multimarginal optimal transport}, Journal of Machine Learning
  Research, 23 (2022), pp.~1--43, \url{http://jmlr.org/papers/v23/19-843.html}.

\bibitem{lucet1997faster}
{\sc Y.~Lucet}, {\em Faster than the fast {L}egendre transform, the linear-time
  {L}egendre transform}, Numerical Algorithms, 16 (1997), pp.~171--185.

\bibitem{Maddalena2003Irrigation}
{\sc F.~Maddalena, S.~Solimini, and J.~Morel}, {\em A variational model of
  irrigation patterns}, Interfaces Free Bound., 5 (2003), pp.~391--415,
  \url{https://doi.org/10.4171/IFB/85}.

\bibitem{Merigot2016Minimal}
{\sc Q.~M\'{e}rigot and J.~Mirebeau}, {\em Minimal geodesics along
  volume-preserving maps, through semidiscrete optimal transport}, SIAM J.
  Numer. Anal., 54 (2016), pp.~3465--3492,
  \url{https://doi.org/10.1137/15M1017235}.

\bibitem{Neufeld2022Numerical}
{\sc A.~Neufeld and Q.~Xiang}, {\em Numerical method for feasible and
  approximately optimal solutions of multi-marginal optimal transport beyond
  discrete measures}, arXiv:2203.01633,  (2022).

\bibitem{mmot_github}
{\sc M.~Parno and B.~Zhou}, {\em {MMOT2d}}.
\newblock \url{https://github.com/simda-muri/mmot}, 2022.

\bibitem{Parno2019Remote}
{\sc M.~D. Parno, B.~A. West, A.~J. Song, T.~S. Hodgdon, and D.~T. O'Connor},
  {\em Remote measurement of sea ice dynamics with regularized optimal
  transport}, Geophysical Research Letters, 46 (2019), pp.~5341--5350,
  \url{https://doi.org/10.1029/2019GL083037}.

\bibitem{Pass2011Uniqueness}
{\sc B.~Pass}, {\em Uniqueness and {M}onge solutions in the multimarginal
  optimal transportation problem}, SIAM J. Math. Anal., 43 (2011),
  pp.~2758--2775, \url{https://doi.org/10.1137/100804917}.

\bibitem{Pass2015Survey}
{\sc B.~Pass}, {\em Multi-marginal optimal transport: theory and applications},
  ESAIM Math. Model. Numer. Anal., 49 (2015), pp.~1771--1790,
  \url{https://doi.org/10.1051/m2an/2015020}.

\bibitem{Peletier2009Partial}
{\sc M.~Peletier and M.~R\"{o}ger}, {\em Partial localization, lipid bilayers,
  and the elastica functional}, Arch. Ration. Mech. Anal., 193 (2009),
  pp.~475--537, \url{https://doi.org/10.1007/s00205-008-0150-4}.

\bibitem{Rockafellar1970Convex}
{\sc R.~T. Rockafellar}, {\em Convex analysis}, Princeton Mathematical Series,
  No. 28, Princeton University Press, Princeton, N.J., 1970.

\bibitem{Santambrogio2015OTAM}
{\sc F.~Santambrogio}, {\em Optimal transport for applied mathematicians},
  vol.~87 of Progress in Nonlinear Differential Equations and their
  Applications, Birkh\"{a}user, 2015,
  \url{https://doi.org/10.1007/978-3-319-20828-2}.
\newblock Calculus of variations, PDEs, and modeling.

\bibitem{Saumier2015OT}
{\sc L.~Saumier, B.~Khouider, and M.~Agueh}, {\em Optimal transport for
  particle image velocimetry: real data and postprocessing algorithms}, SIAM J.
  Appl. Math., 75 (2015), pp.~2495--2514,
  \url{https://doi.org/10.1137/140988814}.

\bibitem{Schmitzer2019Stabilized}
{\sc B.~Schmitzer}, {\em Stabilized sparse scaling algorithms for entropy
  regularized transport problems}, SIAM J. Sci. Comput., 41 (2019),
  pp.~A1443--A1481, \url{https://doi.org/10.1137/16M1106018}.

\bibitem{sinkhorn1967concerning}
{\sc R.~Sinkhorn and P.~Knopp}, {\em Concerning nonnegative matrices and doubly
  stochastic matrices}, Pacific J. Math., 21 (1967), pp.~343--348,
  \url{http://projecteuclid.org/euclid.pjm/1102992505}.

\bibitem{Solomon2015Convolutional}
{\sc J.~Solomon, F.~De~Goes, G.~Peyr{\'e}, M.~Cuturi, A.~Butscher, A.~Nguyen,
  T.~Du, and L.~Guibas}, {\em Convolutional {W}asserstein distances: efficient
  optimal transportation on geometric domains}, ACM Transactions on Graphics
  (ToG), 34 (2015), pp.~1--11, \url{https://doi.org/10.1145/2766963}.

\bibitem{Villani2003TOT}
{\sc C.~Villani}, {\em Topics in optimal transportation}, vol.~58 of Graduate
  Studies in Mathematics, American Mathematical Society, Providence, RI, 2003,
  \url{https://doi.org/10.1090/gsm/058}.

\bibitem{Xia2003Path}
{\sc Q.~Xia}, {\em Optimal paths related to transport problems}, Commun.
  Contemp. Math., 5 (2003), pp.~251--279,
  \url{https://doi.org/10.1142/S021919970300094X}.

\bibitem{Xia2021Existence}
{\sc Q.~Xia and B.~Zhou}, {\em The existence of minimizers for an isoperimetric
  problem with {W}asserstein penalty term in unbounded domains}, Advances in
  Calculus of Variations,  (2021), \url{https://doi.org/10.1515/acv-2020-0083}.

\bibitem{Yang2018Application}
{\sc Y.~Yang, B.~Engquist, J.~Sun, and B.~F. Hamfeldt}, {\em Application of
  optimal transport and the quadratic {W}asserstein metric to full-waveform
  inversion}, Geophysics, 83 (2018), pp.~R43--R62,
  \url{https://doi.org/10.1190/geo2016-0663.1}.

\end{thebibliography}
\end{document}